\documentclass[11pt]{amsart}
\usepackage{ifthen}
\newcommand\mtop{1in}
\newcommand\mbottom{1in}
\newcommand\mleft{1.2in}
\newcommand\mright{1.2in}
\usepackage{amssymb}
\usepackage{amsthm}
\usepackage{stmaryrd}
\usepackage{wasysym}
\usepackage{mathrsfs}
\usepackage[dvipsnames,svgnames,x11names]{xcolor}
\usepackage{hyperref}
\hypersetup{colorlinks=true,linkcolor=RoyalBlue,citecolor=RoyalBlue}
\usepackage{dsfont}
\usepackage[all]{xy}
\SelectTips{cm}{10}  % Better arrowheads; shamelessly stolen from Anton Geraschenko's notes
\usepackage{float}
\providecommand{\mparwidth}{1in}
\providecommand{\mtop}{1in}
\providecommand{\mbottom}{1in}
\providecommand{\mleft}{1.2in}
\providecommand{\mright}{1.2in}
\usepackage[top = \mtop, bottom = \mbottom, left = \mleft, right=\mright, marginparwidth=\mparwidth]{geometry}
\usepackage{fancyhdr}
\usepackage{setspace}
\usepackage{mathtools}
\usepackage{scalefnt}
\usepackage{microtype}
\usepackage{pifont}
\usepackage{ifthen}
\usepackage{makecell}
\usepackage{marginnote}

\usepackage{enumitem}
\setlist[enumerate,1]{leftmargin=6ex,topsep=-5pt}
\setlist[enumerate]{itemsep=7pt}
\setlist[itemize,1]{leftmargin=4ex,topsep=0em}
\setlist[itemize]{itemsep=5pt}
\setlist[itemize,2]{label=$\circ$}
\setlist[itemize,3]{label={\scalefont{0.6}\color{gray}$\blacktriangleright$}}
\setlist[itemize,4]{label=$\ast$}
\setlist{nolistsep}

\usepackage{mdframed}
\usepackage{tikz}
\usetikzlibrary{cd}
\usetikzlibrary{decorations.pathmorphing}
\usetikzlibrary{arrows, decorations.markings}

\usepackage[nolabel]{showlabels} % uncomment (and comment out the next line) to remove \label{} labels
\usepackage{soul}
\usepackage[style=alphabetic,backend=bibtex,doi=false,isbn=false,url=false]{biblatex}
\begin{document}

\newcommand{\theoremnumstyle}{section}
%% Actual formatting
\parskip=0.2in \parindent=0in
\allowdisplaybreaks % Allow align environments to split over page breaks
\raggedbottom % Don't leave awkward spaces in the middle of pages

% From Atanas Atanasov: \replacecommand is similar to \providecommand but it initializes the command as necessary even if a previous definition exists.
\newcommand*{\replacecommand}[1]{%
  \providecommand{#1}{}%
  \renewcommand{#1}%
}

\renewcommand{\arraystretch}{1.5}

%% Miscellaneous \newcommand's
\renewcommand{\l}{\overset}
\newcommand{\into}{\hookrightarrow}
\newcommand{\onto}{\twoheadrightarrow}
\newcommand{\tto}{\longrightarrow}
\newcommand{\too}[1]{\l{#1}\to}
\newcommand{\ttoo}[1]{\l{#1}\longrightarrow}
\newcommand{\intoo}[1]{\l{#1}\into}
\newcommand{\ontoo}[1]{\l{#1}\onto}
\newcommand{\mapstoo}[1]{\l{#1}\mapsto}
\newcommand{\bto}{\leftarrow}
\newcommand{\btto}{\longleftarrow}
\newcommand{\btoo}[1]{\l{#1}\bto}
\newcommand{\bttoo}[1]{\l{#1}\longleftarrow}
\newcommand{\binto}{\hookleftarrow}
\newcommand{\bonto}{\twoheadleftarrow}
\newcommand{\bintoo}[1]{\l{#1}\binto}
\newcommand{\bontoo}[1]{\l{#1}\bonto}
% stupid hack to get the open immersion symbol; see http://tex.stackexchange.com/questions/66723/nudging-overset-characters-downwards
\newcommand{\ointo}{\hspace{3pt}\text{\raisebox{-1.5pt}{$\overset{\circ}{\vphantom{}\smash{\text{\raisebox{1.5pt}{$\into$}}}}$}}\hspace{3pt}}
\newcommand{\lu}{\underset}
\newcommand{\bimplies}{\impliedby}
\newcommand{\ints}{\cap}
\newcommand{\intss}{\bigcap}
\newcommand{\union}{\cup}
\newcommand{\unions}{\bigcup}
\newcommand{\djunion}{\sqcup}
\newcommand{\djunions}{\bigsqcup}
\newcommand{\propersubset}{\subsetneq}
\newcommand{\propersupset}{\supsetneq}
\newcommand{\contains}{\supseteq}
\newcommand{\semidirect}{\rtimes}
\newcommand{\isom}{\cong}
\newcommand{\normal}{\triangleleft}
\replacecommand{\dsum}{\oplus}
\newcommand{\dsums}{\bigoplus}
\newcommand{\tensor}{\otimes}
\newcommand{\tensors}{\bigotimes}
\newcommand{\cotensor}{{\,\scriptstyle\square}}
\let\originalbar\bar
\renewcommand{\bar}[1]{{\overline{#1}}}
\newcommand{\rlim}{\mathop{\varinjlim}\limits}
\newcommand{\llim}{\mathop{\varprojlim}\limits}
\newcommand{\x}{\times}
\replacecommand{\st}{\hspace{2pt} : \hspace{2pt}} % used to be \providecommand
\newcommand{\vv}{\vspace{10pt}}
\newcommand{\til}{\widetilde}
\renewcommand{\hat}{\widehat}
\newcommand{\hhat}{\wedge}
\newcommand{\iy}{\infty}
\newcommand{\hteq}{\simeq}
\newcommand{\dd}[2]{\frac{\partial #1}{\partial #2}}
\newcommand{\sm}{\wedge} % stop getting confused about the word "wedge"
\newcommand{\noqed}{\renewcommand{\qedsymbol}{}}
\newcommand{\adjoint}{\dashv}
\newcommand{\wreath}{\wr}
\newcommand{\heart}{\heartsuit}

%http://tex.stackexchange.com/questions/23432/how-to-create-my-own-math-operator-with-limits
\newcommand{\bigast}{\mathop{\vphantom{\sum}\mathchoice%
  {\vcenter{\hbox{\huge *}}}
  {\vcenter{\hbox{\Large *}}}{*}{*}}\displaylimits}

%% \operatorname
\renewcommand{\dim}{\operatorname{dim}}
\newcommand{\diam}{\operatorname{diam}}
\newcommand{\coker}{\operatorname{coker}}
\newcommand{\im}{\operatorname{im}}
\newcommand{\disc}{\operatorname{disc}}
\newcommand{\Pic}{\operatorname{Pic}}
\newcommand{\Der}{\operatorname{Der}}
\newcommand{\ord}{\operatorname{ord}}
\newcommand{\nil}{\operatorname{nil}}
\newcommand{\rad}{\operatorname{rad}}
\newcommand{\ssum}{\operatorname{sum}}
\newcommand{\codim}{\operatorname{codim}}
\newcommand{\cchar}{\operatorname{char}}
\newcommand{\sspan}{\operatorname{span}}
\newcommand{\rank}{\operatorname{rank}}
\newcommand{\Aut}{\operatorname{Aut}}
\newcommand{\Out}{\operatorname{Out}}
\newcommand{\Div}{\operatorname{Div}}
\newcommand{\Gal}{\operatorname{Gal}}
\newcommand{\Hom}{\operatorname{Hom}}
\newcommand{\Mor}{\operatorname{Mor}}
\newcommand{\Vect}{\operatorname{Vect}}
\newcommand{\Fun}{\operatorname{Fun}}
\newcommand{\Iso}{\operatorname{Iso}}
\newcommand{\Map}{\operatorname{Map}}
\newcommand{\Ho}{\operatorname{Ho}}
\newcommand{\Mod}{\operatorname{Mod}}
\newcommand{\Tot}{\operatorname{Tot}}
\newcommand{\cofib}{\operatorname{cofib}}
\newcommand{\fib}{\operatorname{fib}}
\newcommand{\hocofib}{\operatorname{hocofib}}
\newcommand{\hofib}{\operatorname{hofib}}
\newcommand{\Maps}{\operatorname{Maps}}
\newcommand{\Sym}{\operatorname{Sym}}
\newcommand{\Diff}{\operatorname{Diff}}
\newcommand{\Tr}{\operatorname{Tr}}
\newcommand{\Frac}{\operatorname{Frac}}
\renewcommand{\Re}{\operatorname{Re}} % don't like the curly default ones
\renewcommand{\Im}{\operatorname{Im}}
\newcommand{\gr}{\operatorname{gr}}
\newcommand{\tr}{\operatorname{tr}}
\newcommand{\End}{\operatorname{End}}
\newcommand{\Mat}{\operatorname{Mat}}
\newcommand{\Proj}{\operatorname{Proj}}
\newcommand{\Th}{\operatorname{Thom}}
\newcommand{\Thom}{\operatorname{Thom}}
\newcommand{\Spec}{\operatorname{Spec}}
\newcommand{\Ext}{\operatorname{Ext}}
\newcommand{\Cotor}{\operatorname{Cotor}}
\newcommand{\Tor}{\operatorname{Tor}}
\newcommand{\vol}{\operatorname{vol}}

%% new operatornames since preamble8
\newcommand{\Set}{\operatorname{Set}}
\newcommand{\Top}{\operatorname{Top}}
\newcommand{\Fin}{\operatorname{Fin}}
\newcommand{\Spaces}{\operatorname{Spaces}}
\newcommand{\Sp}{\operatorname{Sp}}
\newcommand{\Spectra}{\operatorname{Spectra}}
\newcommand{\Spt}{\operatorname{Spt}}
\newcommand{\Comod}{\operatorname{Comod}}
\newcommand{\Spf}{\operatorname{Spf}}
\newcommand{\tmf}{\mathit{tmf}}
\newcommand{\Tmf}{\mathit{Tmf}}
\newcommand{\TMF}{\mathit{TMF}}
\newcommand{\Null}{\operatorname{Null}}
\newcommand{\Fil}{\operatorname{Fil}}
\newcommand{\Sq}{\operatorname{Sq}}
\newcommand{\Stable}{\operatorname{Stable}}
\newcommand{\Poly}{\operatorname{Poly}}
\newcommand{\Cat}{\operatorname{Cat}}
\newcommand{\Orb}{\operatorname{Orb}}
\newcommand{\Exc}{\operatorname{Exc}}
\newcommand{\Part}{\operatorname{Part}}
\newcommand{\Comm}{\operatorname{Comm}}
\newcommand{\Res}{\operatorname{Res}}
\newcommand{\Thick}{\operatorname{Thick}}
\newcommand{\red}{{\operatorname{red}}}

\newcommand{\mmf}{\mathit{mmf}}
\newcommand{\Sm}{\operatorname{Sm}}
\newcommand{\Var}{\operatorname{Var}}
\newcommand{\Frob}{\operatorname{Frob}}
\newcommand{\Rep}{\operatorname{Rep}}
\newcommand{\Ch}{\operatorname{Ch}}
\newcommand{\Shv}{\operatorname{Shv}}
\newcommand{\Corr}{\operatorname{Corr}}
\newcommand{\Span}{\operatorname{Span}}
\newcommand{\Sch}{\operatorname{Sch}}
\newcommand{\ev}{\operatorname{ev}}
\newcommand{\Homog}{\operatorname{Homog}}
\newcommand{\conn}{\operatorname{conn}}
\newcommand{\type}{\operatorname{type}}
\newcommand{\num}{\operatorname{num}}
\newcommand{\Aff}{\operatorname{Aff}}
\newcommand{\Psh}{\operatorname{Psh}}
\newcommand{\sk}{\operatorname{sk}}
\newcommand{\cosk}{\operatorname{cosk}}
\newcommand{\Cart}{\operatorname{Cart}}

\newcommand{\Br}{\operatorname{Br}}
\newcommand{\BW}{\operatorname{BW}}
\newcommand{\Cl}{\operatorname{Cl}}
\newcommand{\Conf}{\operatorname{Conf}}
\newcommand{\Alg}{\operatorname{Alg}}
\newcommand{\CAlg}{\operatorname{CAlg}}
\newcommand{\Lie}{\operatorname{Lie}}
\newcommand{\Coalg}{\operatorname{Coalg}}
\newcommand{\Ab}{\operatorname{Ab}}
\newcommand{\Ind}{\operatorname{Ind}}
\newcommand{\ind}{\operatorname{ind}}
\newcommand{\Fix}{\operatorname{Fix}}
\newcommand{\ho}{\operatorname{ho}}
\newcommand{\coeq}{\operatorname{coeq}}
\newcommand{\CMon}{\operatorname{CMon}}
\newcommand{\Sing}{\operatorname{Sing}}
\newcommand{\Inj}{\operatorname{Inj}}
\newcommand{\StMod}{\operatorname{StMod}}
\newcommand{\Loc}{\operatorname{Loc}}
\newcommand{\Free}{\operatorname{Free}}
\newcommand{\Art}{\operatorname{Art}}
\newcommand{\Gpd}{\operatorname{Gpd}}
\newcommand{\Def}{\operatorname{Def}}
\newcommand{\Hyp}{\operatorname{Hyp}}
\newcommand{\Pre}{\operatorname{Pre}}
\newcommand{\Lat}{\operatorname{Lat}}
\newcommand{\Coords}{\operatorname{Coords}}
\newcommand{\cone}{\operatorname{cone}}
\newcommand{\Spc}{\operatorname{Spc}}
\newcommand{\QCoh}{\operatorname{QCoh}}
\newcommand{\height}{\operatorname{ht}}
\newcommand{\Sub}{\operatorname{Sub}}
\newcommand{\Cone}{\operatorname{Cone}}
\newcommand{\Cocone}{\operatorname{Cocone}}
\newcommand{\Ran}{\operatorname{Ran}}
\newcommand{\Lan}{\operatorname{Lan}}
\newcommand{\LieAlg}{\operatorname{LieAlg}}
\newcommand{\Com}{\operatorname{Com}}
\newcommand{\CoAlg}{\operatorname{CoAlg}}
\newcommand{\Prim}{\operatorname{Prim}}
\newcommand{\Coh}{\operatorname{Coh}}
\newcommand{\FormalGrp}{\operatorname{FormalGrp}}
\newcommand{\Fact}{\operatorname{Fact}}
\renewcommand{\Bar}{\operatorname{Bar}}
\newcommand{\Cobar}{\operatorname{Cobar}}
\newcommand{\Ad}{\operatorname{Ad}}
\newcommand{\Moduli}{\operatorname{Moduli}}
\newcommand{\dgla}{\operatorname{dgla}}
\newcommand{\obl}{\operatorname{obl}}
\newcommand{\ob}{\operatorname{ob}}
\newcommand{\IndCoh}{\operatorname{IndCoh}}
\newcommand{\Cocomm}{\operatorname{Cocomm}}
\newcommand{\PreStk}{\operatorname{PreStk}}
\newcommand{\FormGrp}{\operatorname{FormGrp}}
\newcommand{\FormMod}{\operatorname{FormMod}}
\newcommand{\Grp}{\operatorname{Grp}}
\newcommand{\CommAlg}{\operatorname{CommAlg}}
\newcommand{\CoComm}{\operatorname{CoComm}}
\newcommand{\IndSch}{\operatorname{IndSch}}
\newcommand{\Dist}{\operatorname{Dist}}
\newcommand{\Triv}{\operatorname{Triv}}
\newcommand{\Oper}{\operatorname{Oper}}
\newcommand{\Bij}{\operatorname{Bij}}
\newcommand{\Syl}{\operatorname{Syl}}
\newcommand{\Inn}{\operatorname{Inn}}
\newcommand{\Emb}{\operatorname{Emb}}
\newcommand{\Gr}{\operatorname{Gr}}
\newcommand{\CRing}{\operatorname{CRing}}
\newcommand{\sSet}{\operatorname{sSet}}
\newcommand{\et}{\text{\'et}}
\newcommand{\Sh}{\operatorname{Sh}}
\newcommand{\Nil}{\operatorname{Nil}}
\newcommand{\Cech}{\v Cech}
\newcommand{\Stacks}{\operatorname{Stacks}}
\newcommand{\Pin}{\operatorname{Pin}}
\newcommand{\sgn}{\operatorname{sgn}}
%END

%% \A etc.
\newcommand{\A}{\mathbb{A}}
\replacecommand{\C}{\mathbb{C}}
\newcommand{\CP}{\mathbb{C}\mathrm{P}}
\newcommand{\E}{\mathbb{E}}
\newcommand{\F}{\mathbb{F}}
\replacecommand{\G}{\mathbb{G}}
\renewcommand{\H}{\mathbb{H}}
\newcommand{\K}{\mathbb{K}}
\newcommand{\M}{\mathbb{M}}
\newcommand{\N}{\mathbb{N}}
\renewcommand{\P}{\mathbb{P}}
\newcommand{\Q}{\mathbb{Q}}
\newcommand{\R}{\mathbb{R}}
\newcommand{\RP}{\mathbb{R}\mathrm{P}}
\newcommand{\V}{\vee}
\newcommand{\T}{\mathbb{T}}
\providecommand{\U}{\mathscr{U}}
\newcommand{\Z}{\mathbb{Z}}
\newcommand{\e}{\mathfrak{g}}
\newcommand{\m}{\mathfrak{m}}
\newcommand{\n}{\mathfrak{n}}
\newcommand{\p}{\mathfrak{p}}
\newcommand{\q}{\mathfrak{q}}
\renewcommand{\t}{\mathfrak{t}}

%% Large parentheses, etc.
\newcommand{\pa}[1]{\left( {#1} \right)}
\newcommand{\br}[1]{\left[ {#1} \right]}
\newcommand{\cu}[1]{\left\{ {#1} \right\}}
\newcommand{\ab}[1]{\left| {#1} \right|}
\newcommand{\an}[1]{\left\langle {#1}\right\rangle}
\newcommand{\fl}[1]{\left\lfloor {#1}\right\rfloor}
\newcommand{\ceil}[1]{\left\lceil {#1}\right\rceil}
\newcommand{\tf}[1]{{\textstyle{#1}}}
\newcommand{\patf}[1]{\pa{\textstyle{#1}}}

%% Weird constructions
\renewcommand{\mp}{\ \raisebox{5pt}{\text{\rotatebox{180}{$\pm$}}}\ }
\renewcommand{\d}[1]{\ss \mathrm{d}#1}
\newcommand{\imod}{\hspace{-7pt}\pmod}
%\renewcommand{\check}[1]{\overset{\smile}{#1}}

%% Better versions of existing commands
\renewcommand{\epsilon}{\varepsilon}
\renewcommand{\phi}{{\mathchoice{\raisebox{2pt}{\ensuremath\varphi}}{\raisebox{2pt}{\!\! \ensuremath\varphi}}{\raisebox{1pt}{\scriptsize$\varphi$}}{\varphi}}}
\newcommand{\ph}{{\color{white}.\!}}
\newcommand{\tspacer}{{\ensuremath{\color{white}\Big|\!}}}
\newcommand{\chii}{\raisebox{2pt}{\ensuremath\chi}}

% from http://mbork.pl/2009-04-27_Fun_with_quantifiers_%28en%29
\let\originalchi=\chi
\renewcommand{\chi}{{\!{\mathchoice{\raisebox{2pt}{
$\originalchi$}}{\!\raisebox{2pt}{
$\originalchi$}}{\raisebox{1pt}{\scriptsize$\originalchi$}}{\originalchi}}}}

\let\originalforall=\forall
\renewcommand{\forall}{\ \originalforall}

\let\originalexists=\exists
\renewcommand{\exists}{\ \originalexists}

\let\realcheck\check
\newcommand{\vH}{\realcheck{H}}

% quote block: \begin{qu}{leftmargin}{rightmargin}{  ...  }
\newenvironment{qu}[2]
{\begin{list}{}
	  {\setlength\leftmargin{#1}
	  \setlength\rightmargin{#2}}
	  \item[]\footnotesize}
		  {\end{list}}

\newenvironment{titleblock}
{\begin{mdframed}[linecolor=black!20,backgroundcolor=black!15]\begin{center}}
{\end{center}\end{mdframed}}

\newenvironment{shadedblock}[1][5in]
{\bigskip\begin{mdframed}[align=center,userdefinedwidth=#1,linecolor=white,backgroundcolor=black!5]}{\end{mdframed}}

\newenvironment{shadedtitleblock}[2][5in]
{\begin{mdframed}[align=center,userdefinedwidth=#1,linecolor=white,backgroundcolor=black!15]\sc #2\end{mdframed}\begin{mdframed}[align=center,userdefinedwidth=#1,linecolor=white,backgroundcolor=black!5]}{\end{mdframed}}

\newcommand{\shadedheader}[1]{\vspace{25pt}\begin{mdframed}[linecolor=black!20,backgroundcolor=black!5]\sc #1\end{mdframed}\vspace{10pt}}

%%%%%%%%%%%%%%%%%%%%%%%%%
%% newcommands that require certain packages

\newcommand{\itext}{\shortintertext} % requires mathtools
\makeatletter
\@ifundefined{resetu}{ % revert to \u as text accent by defining \resetu
	\renewcommand{\u}{\underbracket[0.7pt]} % requires mathtools
}
\makeatother
\makeatletter
\@ifundefined{resetO}{ % revert to \O as text accent by defining \resetO
	\renewcommand{\O}{\mathcal{O}}
}
\makeatother
\makeatletter
\@ifundefined{resetk}{ % revert to \O as text accent by defining \resetO
	\renewcommand{\k}{\Bbbk}
}
\makeatother

% Colors
% Now this preamble doesn't cause errors when there is no \usepackage{color}
% (you just can't use these commands)
\makeatletter
\@ifundefined{mathds}{
	\newcommand{\Id}{Id}
	}{
	\newcommand{\Id}{\mathds{1}} % requires mathds
	}
\@ifundefined{sethlcolor}{
	\newcommand{\fixmehl}[2]{\underline{#1}\marginpar{\raggedright\smaller\smaller #2}}
	}{\@ifundefined{marginnote}{
		\newcommand{\highlight}[1]{\ifmmode{\text{\sethlcolor{llgray}\hl{$#1$}}}\else{\sethlcolor{llred}\hl{#1}}\fi}
		\newcommand{\fixmehl}[2]{\highlight{#1}\marginpar{\raggedright\smaller\smaller\color{maroon} #2}}
		}{
		\newcommand{\highlight}[1]{\ifmmode{\text{\sethlcolor{llgray}\hl{$#1$}}}\else{\sethlcolor{llred}\hl{#1}}\fi}
		\newcommand{\fixmehl}[2]{\marginnote{\smaller \smaller\color{maroon} #2}{\highlight{#1}}}
		}
	}
\@ifundefined{xy}{}{
	\SelectTips{cm}{10}  % Better arrowheads
}
\@ifundefined{color}{}{
	\definecolor{darkgreen}{RGB}{0,70,0}
	\definecolor{dgreen}{RGB}{0,100,0}
	\definecolor{purple}{RGB}{120,00,120}
	\definecolor{gray}{RGB}{100,100,100}
	\definecolor{mgreen}{RGB}{0,150,0}
	\definecolor{dgreen}{RGB}{0,100,0}
	\definecolor{llgray}{RGB}{230,230,230}
	\definecolor{llred}{RGB}{237,228,228}
	\definecolor{lgreen}{RGB}{100,200,100}
	\definecolor{mgray}{RGB}{150,150,150}
	\definecolor{lgray}{RGB}{190,190,190}
	\definecolor{maroon}{RGB}{150,0,0}
	\definecolor{lblue}{RGB}{120,170,200}
	\definecolor{mblue}{RGB}{65,105,225}
	\definecolor{dblue}{RGB}{0,56,111}
	\definecolor{orange}{RGB}{255,165,0}
	\definecolor{brown}{RGB}{177,84,15}
	\definecolor{rose}{RGB}{135,0,52}
	\definecolor{gold}{RGB}{177,146,87}
	\definecolor{dred}{RGB}{135,19,19}
	\definecolor{mred}{RGB}{194,28,28}
	\newcommand{\edit}[1]{{\it{\color{gray}#1}}}
	\newcommand{\fixme}[1]{{\color{maroon}\it{#1}}}
	\newcommand{\citeme}[2][\!\!]{{\color{orange}[#2~\textit{#1}]}}
	\newcommand{\later}[1]{{\color{dgreen}#1}}
	\newcommand{\corr}[1]{{\color{red}\itshape #1}}
	\newcommand{\question}[1]{\itshape{\color{blue}#1}\upshape}
}
\@ifundefined{substack}{}{
    \newcommand{\attop}[1]{{\let\textstyle\scriptstyle\let\scriptstyle\scriptscriptstyle\substack{#1}}}
    \renewcommand{\atop}[1]{{\let\scriptstyle\textstyle\let\scriptscriptstyle\scriptstyle\substack{#1}}}
}
\makeatother

\newcommand{\tabentry}[1]{\renewcommand{\arraystretch}{1}\begin{tabular}{c}#1\end{tabular}}

\newcommand{\margin}[1]{\marginpar{\raggedright \scalefont{0.7}#1}} % requires scalefnt

% require package pigpen and xy
\newcommand{\pullback}{\ar@{}[rd]|<<{\text{\pigpenfont A}}}
\newcommand{\pushout}{\ar@{}[rd]|>>{\text{\pigpenfont I}}}

% requires xy
% Somehow these don't work when placed earlier in the preamble. Maybe something
% about the \makeatletter thing
\newcommand{\longleftrightarrows}{\xymatrix@1@C=16pt{
\ar@<0.4ex>[r] & \ar@<0.4ex>[l]
}}
\newcommand{\longrightrightarrows}{\xymatrix@1@C=16pt{
\ar@<0.4ex>[r]\ar@<-0.4ex>[r] & 
}}
\newcommand{\mapstto}{\,\xymatrix@1@C=16pt{
\ar@{|->}[r] & 
}\,}
\newcommand{\mapsttoo}[1]{\xymatrix@1@C=16pt{
\ar@{|->}[r]^-{#1} & 
}}
\newcommand{\rightrightrightarrows}{\xymatrix@1@C=16pt{
\ar[r]\ar@<0.8ex>[r]\ar@<-0.8ex>[r] & 
}}
\newcommand{\longleftleftarrows}{\xymatrix@C=16pt{
 & \ar@<0.4ex>[l]\ar@<-0.4ex>[l]
}}
\newcommand{\leftleftleftarrows}{\xymatrix@1@C=16pt{
 & \ar[l]\ar@<0.8ex>[l]\ar@<-0.8ex>[l]
}}
\newcommand{\leftleftleftleftarrows}{\xymatrix@1@C=16pt{
 & \ar@<0.8ex>[l]\ar@<0.3ex>[l]\ar@<-0.3ex>[l]\ar@<-0.8ex>[l]
}}
\newcommand{\lcircle}{\ar@(ul,dl)} % arrow circle to the left of the node
\newcommand{\rcircle}{\ar@(ur,dr)}
\newcommand{\intto}{\ \xymatrix@1@C=16pt{
\ar@{^(->}[r] & 
}}

\newcommand{\eva}[1]
   {{\color{Periwinkle}\it #1}}
\newcommand{\EvaNote}[1]
   {\marginpar{\raggedright\smaller \smaller \color{Periwinkle}#1}}

\makeatletter
\newcommand{\switchmargin}{
\if@reversemargin
\normalmarginpar
\else
\reversemarginpar
\fi
}
\makeatother
\newcommand{\highlighteva}[1]{\ifmmode{\text{\sethlcolor{Lavender}\hl{$#1$}}}\else{\sethlcolor{Lavender}\hl{#1}}\fi}
\newcommand{\EvaNoteHl}[2]{\marginnote{\smaller \smaller \color{Periwinkle}
#2}{\highlighteva{#1}}\switchmargin}

%  vim:ft=tex

\newtheoremstyle{gloss}{\topsep}{\topsep}{}{0pt}{\bfseries}{}{\newline}{\newline *{\bf #3} }
\theoremstyle{gloss}
\newtheorem*{defstar}{Definition}

\newtheoremstyle{newplain}{20pt}{0pt}{\it}{0pt}{\bfseries}{.}{1ex}{}
\theoremstyle{newplain}

% Number by section (1.1) by default, can override with
% \newcommand{\theoremnumstyle}{} (must be empty)
\ifthenelse{\isundefined\theoremnumstyle}
	{\newtheorem{theorem}{Theorem}[section] 
	\numberwithin{equation}{section}} % Number equations by section, like (1.1) instead of (1)
	{\ifthenelse{\equal\theoremnumstyle{}}
		{\newtheorem{theorem}{Theorem}}
		{\newtheorem{theorem}{Theorem}[section]
		\numberwithin{equation}{section}
		}
	}

\newtheorem{corollary}[theorem]{Corollary}
\newtheorem{claim}[theorem]{Claim}
\newtheorem{lemma}[theorem]{Lemma}
\newtheorem{proposition}[theorem]{Proposition}
\newtheorem{fact}[theorem]{Fact}

\newtheoremstyle{newtextthm}{20pt}{0pt}{}{0pt}{\bfseries}{.}{1ex}{}
\theoremstyle{newtextthm}
\newtheorem{definition}[theorem]{Definition}
\newtheorem{example}[theorem]{Example}
\newtheorem{problem}[theorem]{Problem}
\newtheorem{remark}[theorem]{Remark}
\newtheorem{notation}[theorem]{Notation}

\newtheorem*{theoremstar}{Theorem}
\newtheorem*{lemmastar}{Lemma}
\newtheorem*{corstar}{Corollary}
\newtheorem*{corollarystar}{Corollary}
\newtheorem*{propositionstar}{Proposition}
\newtheorem*{claimstar}{Claim}
\newtheorem*{examplestar}{Example}

% random
\newcommand{\argforrandom}{}
\theoremstyle{newplain}
\newtheorem{helperforrandom}[theorem]{\argforrandom}
\newtheorem*{helperforrandomstar}{\argforrandom}
\newenvironment{random}[1]{\renewcommand{\argforrandom}{#1}\begin{helperforrandom}}{\end{helperforrandom}}
\newenvironment{randomstar}[1]{\renewcommand{\argforrandom}{#1}\begin{helperforrandomstar}}{\end{helperforrandomstar}}

\newenvironment{exercise}[1]{\hspace{1pt}\nn \large {\sc #1.}\hv \normalsize
\vspace{10pt}\\ }{} 
\newcommand{\subthing}[1]{\hv\large(#1)\hv\hv \normalsize }

\newcommand{\itemref}[1]{(\ref{#1})}

\renewcommand{\showlabelfont}{\tiny\tt\color{mgreen}}

\newcommand{\EE}{E^E}
\newcommand{\el}{\textit{ell}}
\newcommand{\Eel}{E^{\el}}

\renewbibmacro{in:}{}

\title{Beta families arising from a $v_2^9$ self map on $S/(3,v_1^8)$}
\author{Eva Belmont and Katsumi Shimomura}
\maketitle

\begin{abstract}
We show that $v_2^9$ is a permanent cycle in the 3-primary Adams-Novikov
spectral sequence computing $\pi_*(S/(3,v_1^8))$, and use this to conclude that
the families $\beta_{9t+3/i}$ for $i=1,2$, $\beta_{9t+6/i}$
for $i=1,2,3$, $\beta_{9t+9/i}$ for $i=1,\dots,8$, $\alpha_1\beta_{9t+3/3}$, and
$\alpha_1\beta_{9t+7}$ are permanent cycles in the 3-primary Adams-Novikov
spectral sequence for the sphere for all $t\geq 0$.
We use a computer program by Wang
to determine the additive and partial multiplicative
structure of the Adams-Novikov $E_2$ page for the sphere in relevant degrees.
The $i=1$ cases recover previously known results of Behrens-Pemmaraju
\cite{behrens-pemmaraju} and the second author \cite{shimomura-note-beta}.
The results about $\beta_{9t+3/3}$, $\beta_{9t+6/3}$ and $\beta_{9t+9/8}$ were
previously claimed by the second author \cite{shimomura-beta}; the computer calculations
allow us to give a more direct proof.
As an application, we determine the image of the Hurewicz map $\pi_*S \to
\pi_*\tmf$ at $p=3$. 
\end{abstract}

\section{Introduction}\label{sec:intro}
Miller, Ravenel, and Wilson showed that the 2-line of the Adams-Novikov
$E_2$ page for the sphere is generated by classes $\beta_{i/j,k}$ for $i,j,k$
satisfying certain conditions \cite[Theorem 2.6]{MRW}.
At the prime 3, the $\beta$ elements with $i\leq 9$ are:
\begin{align*}
 & \beta_i \text{ for $i=1,2,4,5,7,8$};
\\ & \beta_{3/j} \text{ and }\beta_{6/j} \text{ for $j=1,2,3$};
\\ & \beta_{9/j}\text{ for $j=1,\dots,9$};
\\ & \beta_{9/3,2};
\end{align*}
where we write $\beta_{i/j} := \beta_{i/j,1}$ and $\beta_i := \beta_{i/1}$.
They have order 3 except for $\beta_{9/3,2}$, which satisfies $3\beta_{9/3,2}=\beta_{9/3}$.
Of these, the permanent cycles are the following:
\begin{equation}\label{eq:beta-perm} \beta_1,\ \beta_2,\ \beta_{3/2},\ \beta_3,\ \beta_5,\ \beta_{6/3},\
\beta_{6/2},\ \beta_6,\ \beta_7+ c\beta_{9/9} \text{ for some $c\in \{ \pm 1
\}$},\ \beta_{9/j}\text{ for $1\leq j\leq 8$}. \end{equation}
For $i\leq 7$ and $\beta_7+c\beta_{9/9}$, these assertions
can be read off the exhaustive calculation \cite[Table A3.4]{green} of
$\pi_nS^\hhat_3$ in stems $n\leq 108$; see also \cite{oka-beta} for many of the
survival results and
\cite{shimomura-L2-toda-smith} for the non-survival results. The element
$\beta_7+c\beta_{9/9}$ is an Arf invariant class (an odd-primary analogue of the
$p=2$ Kervaire invariant classes), discussed in \cite[p.439]{ravenel-arf}; the
survival of the Arf invariant classes is not known in general at $p=3$. The
survival of $\beta_{9/j}$ for $j\leq 8$ is a consequence of Theorem
\ref{thm:beta}, which does not depend on prior knowledge about this element, but
we do not claim originality for this result.

These results arise from exhaustive calculations in tractable stems, but it is
possible to prove results about $\beta$ elements outside the range of feasible
computation. One strategy is as follows. Suppose $\beta_{i/j}$ is a
permanent cycle. It is $v_1$-power-torsion; that is, there exists a type 2 complex $V$
such that $\beta_{i/j}$ factors as $S\too{f_1}V \too{f_2}S$ (we omit degree
shifts for clarity of notation). If $v_2^t$ is a $v_2$ self map on $V$, then we may
construct elements of $\pi_*(S)$ as follows:
$$ S\ttoo{f_1} V \ttoo{v_2^t} V\ttoo{f_2} S, \hspace{20pt} t\geq 0. $$
For this family to be of interest, one must also show that the elements are
nonzero, for example by identifying their Adams-Novikov representatives.

Let $S/3$ denote the mod 3 Moore space, and for $m\geq 1$ let
$S/(3,v_1^m)$ denote the cofiber of the $m$-fold iterate of Adams' $v_1$ self map
$S/3\ttoo{v_1}S/3$ \cite{toda-realizing}.
Behrens and Pemmaraju \cite{behrens-pemmaraju} show there is a $v_2^9$ self map on
$S/(3,v_1)$ and use this to prove the existence of nonzero homotopy classes
represented by $\beta_{9t+s}$ for $s=1,2,5,6,9$ and $t\geq 0$. The second author
\cite{shimomura-note-beta} proves the existence of $\beta_{9t+3}$. By comparison to
$L_2$-local homotopy \cite{shimomura-L2-toda-smith}, he shows that the elements
$$ \beta_{9t+4},\ \beta_{9t+7},\ \beta_{9t+8},\ \beta_{9t+3/3},\
\beta_{9t/3,2},\ \beta_{3^is/3^i} $$
are not permanent cycles for $t\geq 1$, $s\not\equiv 0\pmod 3$, and $i>1$. The
main goal of this paper is to construct a $v_2^9$ self map on $S/(3,v_1^8)$ and
show the remaining $\beta$ elements in $\pi_s(S)$
for $s\leq |v_2^9| = 144$ also give rise to infinite families.
\begin{randomstar}{Theorem \ref{thm:beta}}
For all $t\geq 0$, the classes
\begin{align*}
\beta_{9t+3/j}  & \text{ for }j=1,2
\\\beta_{9t+6/j} &\text{ for } j=1,2,3
\\\beta_{9t+9/j} & \text{ for } j=1,\dots,8
\\\alpha_1\beta_{9t+3/3} & 
\\\alpha_1\beta_{9t+7} & 
\end{align*}
are permanent cycles in
the Adams-Novikov spectral sequence for the sphere.
\end{randomstar}

These families are interesting in part because the Hurewicz map $\pi_*(S)\to
\pi_*(\tmf)$ detects $\beta_{9t+1}$, $\alpha_1\beta_{9t+3/3}$,
$\beta_{9t+6/3}$, and $\alpha_1\beta_{9t+7}$, as we show in Theorem
\ref{thm:hurewicz-image}. Together with the well-known behavior in the 0- and
1-lines, this completely determines the Hurewicz image of $\tmf$ at $p=3$.
Behrens, Mahowald, and Quigley \cite{behrens-mahowald-quigley} 
calculate the Hurewicz image of $\tmf$ at $p=2$. Since $\pi_*(\tmf[1/6]) =
\Z[1/6,a_4,a_6]$ is concentrated on the Adams-Novikov 0-line, our work together
with the $p=2$ case forms the complete determination of the Hurewicz image of
$\tmf$ at all primes.

Following the strategy outlined above, much of the work involves showing that
$v_2^9$ is a permanent cycle in the Adams-Novikov spectral sequence computing
$\pi_*(S/(3,v_1^8))$ (Theorem \ref{thm:v2^9}).
All of our explicit calculations are in the Adams-Novikov spectral sequence for
$S/3$. To relate this to $S/(3,v_1^8)$ we 
use a lemma due to the second author (see Lemma
\ref{lem:extensions-diff}) that relates $v_1^m$-extensions in the Adams-Novikov spectral
sequence for $S/3$ to differentials in the Adams-Novikov spectral sequence for
$S/(3,v_1^m)$. 
Combined with Oka's result \cite{oka-ring}
that $S/(3,v_1^m)$ is a ring spectrum for $m\geq 2$, this implies the existence of a
$v_2^9$-self-map.

\begin{randomstar}{Corollary \ref{cor:self-map}}
For $2\leq m\leq 8$, there is a nonzero self-map $v_2^9:\Sigma^{144}S/(3,v_1^m)\to
S/(3,v_1^m)$.
\end{randomstar}
There is also a similar result for $m=9$, but correction terms for $v_2^9$ are
needed; see Remark \ref{rmk:m=9}.

Our proof that $v_2^9$ is a permanent cycle in the Adams-Novikov spectral
sequence computing $\pi_*(S/(3,v_1^8))$ relies on analysis of the 143-stem
in the Adams-Novikov spectral sequence for the sphere. This is greatly aided by software written by Wang
\cite{guozhen-github, guozhen-documentation}, which computes the $E_2$ page of the Adams-Novikov
spectral sequence for the sphere using the algebraic Novikov spectral sequence.
In addition,
the software computes multiplication by $p$, $\alpha_1$, and arbitrary
$\beta_{i/j}$ elements.
Wang's software was originally written for use at $p=2$; the minor modifications
we used to change the prime are available at \cite{guozhen3} and data, charts,
and more documentation are available at \cite{guozhen-data}. The calculations
that make use of computer data occur solely in Section \ref{sec:143}.

We now comment on the overlap between this work and the preprint
\cite{shimomura-beta} by the second author: both works construct $v_2^9$ and the
families $\beta_{9t+3/3}$, $\beta_{9t+6/3}$, and $\beta_{9t+9/8}$, but we
find the methods here to be more straightforward. The earlier preprint uses the
machinery of infinite descent to control the complexity of the Adams-Novikov
spectral sequence, while we opt to work directly with the Adams-Novikov $E_2$
page, controlling the complexity using Wang's program. In particular,
our analysis in Section \ref{sec:143}, which is the crucial input for the
construction of $v_2^9$, follows from the $\beta_1$-multiplication structure
given by the computer calculations, as most of the elements in play are highly
$\beta_1$-divisible.

We conclude this section by giving an outline of the rest of the paper.
In Section \ref{sec:notation} we state notational conventions for the rest of
the paper and write down some easy facts about the
Adams-Novikov spectral sequence that will be used extensively in the remaining
sections. Most of the work for proving Theorem \ref{thm:v2^9} occurs in Section
\ref{sec:143}, which makes use of computer calculations to determine the
Adams-Novikov spectral sequence for $S/3$ near the 143 stem.
Theorem \ref{thm:v2^9}, which constructs $v_2^9$, is proved in Section \ref{sec:v2^9}. In Section \ref{sec:betas} we prove Theorem \ref{thm:beta},
which constructs the promised $\beta$ families. This involves explicit
calculations in a tractable range of stems to prove that $v_1^2v_2^3$,
$v_1v_2^6$, and $\alpha_1v_1v_2^3$ in $E_2(S/(3,v_1^4))$ and $\alpha_1v_1v_2^7$ in
$E_2(S/(3,v_1^2))$ are permanent cycles.
In Section \ref{sec:tmf} we determine the 3-primary Hurewicz image of $\tmf$
(Theorem \ref{thm:hurewicz-image}).

\textbf{Acknowledgements:} The first author would like to thank Paul Goerss for
suggesting this project, and for many helpful conversations along
the way. She would also like to thank Guozhen Wang for explaining some aspects
of his program, and for some code changes.

\section{Notation and preliminaries}\label{sec:notation}
At a fixed prime $p$, the Brown-Peterson spectrum $BP$ has coefficient ring
$BP_* = \Z_{(p)}[v_1,v_2,\dots]$ with $|v_i| = 2p^i-2$ and ring of co-operations
$BP_*BP = BP_*[t_1,t_2,\dots]$ with $|t_i| = 2p^i-2$. Given a finite
$p$-local spectrum $X$, the Adams-Novikov spectral sequence
\begin{equation}\label{eq:ANSS} E_2 = \Ext^{*,*}_{BP_*BP}(BP_*, BP_*X)\implies \pi_*(X)_{(p)}
\end{equation}
converges. Henceforth everything will be implicitly localized at the prime $p=3$. The $E_2$
page of \eqref{eq:ANSS} can be calculated as the cohomology of the normalized cobar complex
\begin{equation}\label{eq:cobar} BP_* \ttoo{\eta_R-\eta_L} \bar{BP_*BP} \tto \bar{BP_*BP}^{\tensor
2}\tto \dots \end{equation}
though this is not an efficient means of computation.
See \cite{goerss-ANSS-notes}, \cite[\S4.3,\S4.4]{green} for further background
on the Adams-Novikov spectral sequence.

Let $E_r^{s,f}(X)$ denote the $E_r$ page of \eqref{eq:ANSS}, restricted to stem $s$ and
Adams-Novikov filtration $f$. We say that an element in $\pi_sX$ is detected in
filtration $f$ if it is represented by a nonzero class in $E_\iy^{s,f}(X)$.
Throughout, any equality of homotopy or $E_2$ page
elements should be understood to be true up to units (that is, up to signs).

We will make frequent use of the cofiber sequence
$$ S\too{3} S \too{i} S/3 \too{j} \Sigma S. $$
We will also consider the cofiber sequences
$$ \Sigma^{4m}S/3\ttoo{v_1^m} S/3 \ttoo{i_m} S/(3,v_1^m)\ttoo{j_m} \Sigma^{4m+1} S/3 $$
for $m\geq 1$, eventually focusing primarily on $m=8$.
Henceforth degree shifts in cofiber and long exact sequences will usually not be
shown.
The maps $i$, $j$, $i_m$, and $j_m$ induce maps of Adams-Novikov spectral
sequences, which we will denote with the same letters.
We note the effect on degrees: given $x\in E_2^{s,f}(S/3)$, we have
$j(x)\in E_2^{s-1,f+1}(S)$; given $x\in E_2^{s,f}(S/(3,v_1^m))$, we have
$j_m(x)\in E_2^{s-4m-1,f+1}(S/3)$.
The maps $i$ and $i_m$ preserve degrees.
Sometimes we omit applications of $i$ or $i_m$ in the notation for brevity; for
example, we write $\beta_1\in E_2^{10,2}(S/3)$ to refer to $i(\beta_1)$.
This is justified by regarding $E_2(S/3)$ as a module over $E_2(S)$.

By the Geometric Boundary Theorem (see \cite[Theorem 2.3.4]{green}), the
maps induced by $j$ and $j_m$ on $E_2$ pages coincide with the boundary
maps in the long exact sequences of Ext groups
\begin{align*}
\Ext^{*,*}_{BP_*BP}(BP_*, BP_*/3) & \to \Ext^{*+1,*}_{BP_*BP}(BP_*, BP_*)
\\\Ext^{*,*}_{BP_*BP}(BP_*, BP_*/(3,v_1^m)) & \to \Ext^{*+1,*}_{BP_*BP}(BP_*,
BP_*/3).
\end{align*}

\begin{definition}
We will say that an element $x\in E_2^{*,*}(S/3)$ is a \emph{bottom cell
element} if it is in the image of $i:E_2^{*,*}(S)\to E_2^{*,*}(S/3)$. An
element $x\in E_2^{*,*}(S/3)$ is a \emph{top cell element} if its image under the
boundary map $j: E_2^{*,*}(S/3)\to E_2^{*-1,*}(S)$ is nonzero.
%We will also call a class in $E_2^{*,*}(S/(3,v_1^3))$ a $v_1^3$-bottom cell
%class [respectively, $v_1^3$-top cell class] if it is in the image of
%$i_3$ [respectively, in the kernel of $j_3$].
\end{definition}

\begin{random}{Notation}\label{notation:tilde}
~\begin{enumerate} 
\item If $x\in E_2^{s,f}(S)$ is 3-torsion, we will let $\bar{x}\in E_2^{s+1,f-1}(S/3)$
denote a class such that $j(\bar{x}) = x$. Note that $\bar{x}$
may not always be uniquely determined.
\item If $x\in E_2(X)$ is a permanent cycle converging to $y\in \pi_*(X)$, write
$y = \{ x \}$.
\end{enumerate}
\end{random}

In the rest of this section we present some preliminaries that are important for
working with the Adams-Novikov spectral sequences for $S$, $S/3$, and
$S/(3,v_1^m)$. All of these facts are well-known, and the rest of this section
can be skipped by a knowledgable reader.
First we recall some frequently-encountered
permanent cycles in the 3-primary Adams-Novikov spectral sequence for the
sphere. The comparisons below to the Adams spectral sequence are not needed in the
rest of the paper, and are just presented for those readers who are more
familiar with the Adams elements; a reference for computational facts about the
Adams spectral sequence $E_2$ page is \cite[\S3.4]{green}, and for the
corresponding Adams-Novikov elements is \cite{MRW} or
\cite[\S6]{goerss-ANSS-notes} for the Greek letter construction and \cite[Theorem 4.4.20]{green} for low stems.
\begin{itemize} 
\item $\alpha_1\in E_2^{3,1}(S)$ is represented by $[t_1]$ in the cobar complex
\eqref{eq:cobar}, and is called $h_0$ in the Adams spectral sequence.
\item $\beta_1\in E_2^{10,2}(S)$ equals the Massey product
$\an{\alpha_1,\alpha_1,\alpha_1}$ and is called $b_0=b_{10}$ in the Adams spectral
sequence.
\item $\beta_2\in E_2^{26,2}(S)$ is called $k=k_0=\an{h_0,h_1,h_1}$ in the Adams
spectral sequence. (This does not correspond to an Adams-Novikov Massey
product since $h_1$ does not exist in the Adams-Novikov $E_2$ page.)
\end{itemize}
The 0-line (generated by just $1\in E_2^{*,0}(S)$) and the 1-line $E_2^{*,1}(S)$
consist of the image of the $J$ homomorphism. These classes are all permanent
cycles; the image of the 1-line under $i$ is $\alpha_1v_1^m$ for $m\geq 0$.
The 0-line of $E_2(S/3)$ is the polynomial algebra on $v_1$.

\begin{fact}\label{fact:sparseness}
Let $X = S$, $S/3$, or $S/(3,v_1^m)$ for $m\geq 1$.
\begin{enumerate} 
\item $E_2^{s,f}(X)=0$ if $s+f\not\equiv 0\pmod 4$;
\item $E_2^{s,f}(X) = E_5^{s,f}(X)$.
\end{enumerate}
\end{fact}
\begin{proof}
See \cite[Proposition 4.4.2]{green} for the statement about $X=S$.
This sparseness for the sphere also implies the first statement for $X=S/3$ and
$S/(3,v_1^m)$, as can be seen by looking at the degrees of the long exact
sequences in $\Ext$ groups corresponding to the short exact sequences
$BP_*\too{3} BP_*\to BP_*/3$ and $BP_*/3\too{v_1^m} BP_*/3\to BP_*/(3,v_1^m)$.
The second statement follows from the first.
\end{proof}

Most of our calculations in the Adams-Novikov spectral sequence for
$S/(3,v_1^m)$ for $m\geq 2$ implicitly use the following fact.
\begin{theorem}[\cite{oka-ring}] \label{thm:oka-ring}
For $m\geq 2$, $S/(3,v_1^m)$ is a ring spectrum.
\end{theorem}

It is also well-known that $S/3$ is a ring spectrum.

\begin{lemma}\label{lem:b6}
~\begin{enumerate} 
\item \label{item:b6} If $x\in E_{10}(S)$, then $\beta_1^6x$ is zero in $E_{10}(S)$. If $x\in E_{10}(S/3)$, then $\beta_1^6x$ is zero in $E_{10}(S/3)$.
\item \label{item:ab3} If $x\in E_6(S)$, then $\alpha_1\beta_1^3x$ is zero in
$E_6(S)$. If $x\in E_6(S/3)$, then $\alpha_1\beta_1^3x$ is zero in $E_6(S/3)$.
\item \label{item:v1^2beta} We have $v_1^2\cdot \beta_1 = 0$ in $E_2(S/3)$.
\end{enumerate}
\end{lemma}
\begin{proof}
For \eqref{item:b6}, the classical differential $d_9(\alpha_1\beta_4) = \beta_1^6$
(see Table \ref{tab:anss})
implies that $\beta_1^6=0$ in $E_{10}(S)$, and hence
$\beta_1^6 = 0$ in $E_{10}(S/3)$. Part \eqref{item:ab3} is an analogous consequence of
the Toda differential $d_5(\beta_{3/3}) = \alpha_1\beta_1^3$.
Part \eqref{item:v1^2beta} is \cite[Lemma 2.13]{shimomura-beta-tpr}.
\end{proof}

Next, we record some basic facts about transferring differentials by naturality
across various Adams-Novikov spectral sequences.

\begin{lemma}\label{lem:d5}
Let $m\geq 1$.
\begin{enumerate} 
\item \label{item:i-transfer} If there is a nontrivial differential $d_5(x)=y$ in $E_5(S)$ where $y$
is not 3-divisible, then there is a nontrivial differential
$d_5(i(x))=i(y)$ in $E_5(S/3)$.
\item \label{item:delta-transfer} If there is a nontrivial differential $d_5(j(x)) = j(y)$ in
$E_5(S)$ for $x,y\in E_5(S/3)$, then $d_5(x)\neq 0$, and $d_5(x)\equiv y\pmod
{\Im(i)}$.
\end{enumerate}
\end{lemma}
\begin{proof}
For \eqref{item:i-transfer}, by naturality of $i$, there is a differential $d_5(i(x))=i(y)$. We just need to
check that $i(y)$ is nonzero in $E_5(S/3)$. This follows from the fact that
$E_2(S/3) = E_5(S/3)$ and the assumption that $i(y)$ is nonzero in $E_2(S/3)$.
For \eqref{item:delta-transfer}, since $j$ commutes with the differential and
$E_2(S/3) = E_5(S/3)$, we have that $d_5(x)\equiv y\pmod {\ker(j)}$. The long exact sequence
$$\dots\too{3} E_2(S)\too{i} E_2(S/3) \too{j} E_2(S)\too{3}\dots $$
implies that $\ker(j) = \Im(i)$.
\end{proof}

We use the following lemma without further mention when working with $\beta$
elements, applying it to the case $x = v_2^i$.
\begin{lemma}\label{lem:delta-switch}
Let $m\geq 1$. For any $x\in E_2(S/(3,v_1^m))$ we have $j_m(x)=j_{m+k}(v_1^kx)$.
\end{lemma}
\begin{proof}
The map of short exact sequences
$$ \xymatrix@C=30pt{
BP_*/3\ar[d]_=\ar[r]^-{v_1^m} & BP_*/3\ar[r]\ar[d]^-{v_1^k} & BP_*/(3,v_1^m)\ar[d]^-{v_1^k}
\\BP_*/3\ar[r]^-{v_1^{m+k}} & BP_*/3\ar[r] & BP_*/(3,v_1^{m+k})
}$$
induces a map of long exact sequences after applying $\Ext_{BP_*BP}(BP_*, -)$.
In particular, we have a commutative diagram as follows.
\begin{align*}
 & \begin{gathered}[b]
\xymatrix@C=35pt{
\overbracket[0.5pt]{\Ext_{BP_*BP}(BP_*, BP_*/(3,v_1^m))}^{E_2(S/(3,v_1^m))}\ar[r]^-{j_m}\ar[d]^-{v_1^k}
& \overbracket[0.5pt]{\Ext_{BP_*BP}(BP_*, BP_*/3)}^{E_2(S/3)}\ar[d]^=
\\\Ext_{BP_*BP}(BP_*, BP_*/(3,v_1^{m+k}))\ar[r]^-{j_{m+k}} & \Ext_{BP_*BP}(BP_*,
BP_*/3)
}\\[-\dp\strutbox]
\end{gathered}\qedhere
\end{align*}
\end{proof}

\section{Computer-assisted calculations in the 143-stem}\label{sec:143}
In this section we study the Adams-Novikov spectral sequence for $S/3$ in the
143-stem and nearby stems; this is the main technical input needed for Theorem
\ref{thm:v2^9}. We make use of computer calculations of the Adams-Novikov $E_2$
page for the sphere; the specific facts from the computer data we use are given in Lemma \ref{lem:computer}.
The results from this section that are used later are
Lemma \ref{lem:143-v1div} and Proposition \ref{prop:143-only-5}. The former
follows immediately from the $\F_3$-vector space structure
of $E_2^{*,*}(S)$. The rest of the section is devoted to proving the latter,
which says that every permanent cycle in $\pi_{143}(S/3)$ is detected in
filtration $\leq 5$.
This requires more careful analysis using the multiplicative structure of the
$E_2$ page. Lemmas \ref{lem:141,15} and \ref{lem:142,14} give the differentials
responsible for killing higher filtration elements in $E_2^{143,*}(S/3)$.

We encourage the reader to refer to the Adams-Novikov chart in
\cite{guozhen-data} while reading this section. Table \ref{tab:anss} is a
summary of this data:
all of the differentials in the chart in \cite{guozhen-data} are derived from $\alpha_1$,
$\beta_1$, and $\beta_2$-multiples of the classes in Table \ref{tab:anss}.
Here $x_{57}$ is the generator of $E_2^{57,3}(S)$, $x_{75}$ is the generator of
$E_2^{75,5}(S)$, and $x_{96}$ is the generator of $E_2^{96,4}(S)$.
Moreover, the differentials are complete through stem 108.

\renewcommand{\figurename}{Table}
\begin{figure}[H]
\begin{tabular}{c|c|c|c|c|c}
source $s$ & source $f$ & source & $d_r$ & target & reason
\\\hline\hline 34 & 2 & $\beta_{3/3}$ & $d_5$ & $\alpha_1\beta_1^3$ & Toda differential
\\\hline 57 & 3 & $x_{57}$ & $d_5$ & $\beta_1^3\beta_2$ & forced by \cite[Table A3.4]{green}
\\\hline 58 & 2 & $\beta_4$ & $d_5$ & $\alpha_1\beta_1^2\beta_{3/3}$ & forced by \cite[Table A3.4]{green}
\\\hline 61 & 6 & $\alpha_1\beta_4$ & $d_9$ & $\beta_1^6$ & forced by \cite[Table A3.4]{green}
\\\hline 89 & 3 & nonzero class & $d_5$ & nonzero class &  forced by \cite[Table A3.4]{green}
\\\hline 96 & 4 & $x_{96}$ & $d_5$ & $\beta_1^2x_{75}$ & See Lemma \ref{lem:99,5}
\end{tabular}
\caption{Some classical Adams-Novikov differentials}
\label{tab:anss}
\end{figure}

\begin{lemma}[\cite{guozhen-data,guozhen3}]\label{lem:computer}
~\begin{enumerate} 
\item \label{item:81} $\dim(E_2^{81,3}(S)) = 2$, $E_2^{81,7}(S) =
\F_3\{\alpha_1\beta_1^2\beta_4\}$, and $\dim(E_2^{81,f}(S))=0$ if $f\neq 3,7$.
\item \label{item:96} $E_2^{95,9}(S)=\F_3\{\beta_1^2 x_{75}\}$ where $x_{75}$ is the generator of
$E_2^{75,5}(S)$ and $\alpha_1\beta_1^2x_{75}\neq 0$. The only other generator in
$E_2^{95,\geq 9}(S)$ is $\alpha_1\beta_1^4\beta_2^2\in E_2^{95,13}(S)$.
\item \label{item:99} $\dim(E_2^{99,5}(S)) = 2$, $\dim(\alpha_1E_2^{99,5})=1$,
and one of the generators of $E_2^{99,5}(S)$ is $\alpha_1x_{96}$ where $x_{96}$ is the generator of $E_2^{96,4}(S)$. Moreover,
$E_2^{99,17}(S) = \F_3\{ \alpha_1\beta_1^7\beta_2 \}$ and $E_2^{99,f}(S)=0$ for
$f\neq 5,17$.
\item\label{item:for-v1^2-div} $E_2^{135,5}(S)=0 = E_2^{134,6}(S)$.
\item \label{item:141} $E_2^{141,15}(S) = \beta_1^6 E_2^{81,3}(S)$ and this
group has dimension 2.
\item \label{item:chart} Figure \ref{fig:143-S} displays the vector space structure of
$E_2^{s,f}(S)$ for $140\leq s\leq 144$, as well as selected multiplicative
structure.
\end{enumerate}
\end{lemma}
In Figure \ref{fig:143-S}, the names in $E_2^{141,15}(S)$ follow from the proof
of Lemma \ref{lem:pi81}; other names are multiplications computed using Wang's
program.

\begin{lemma}\label{lem:143-v1div}
If $x\in E_2^{143,5}(S/3)$ is $v_1^2$-divisible, then $x=0$.
\end{lemma}
\begin{proof}
Lemma \ref{lem:computer}\eqref{item:for-v1^2-div} implies $E_2^{135,5}(S/3)=0$.
\end{proof}

\begin{lemma}\label{lem:pi81}
We have that $E_2^{81,3}(S)$ is 2-dimensional, and both generators are permanent
cycles.
\end{lemma}
\begin{proof}
By \cite[Table A3.4]{green}, $\pi_{81}(S)^\hhat_3$ is 2-dimensional, generated
by $\gamma_2$ and $\an{\alpha_1,\alpha_1,\beta_5}$. Lemma
\ref{lem:computer}\eqref{item:81} gives the structure of $E_2^{81,*}(S)$.
It suffices to show that $\alpha_1\beta_1^2\beta_4$ supports a nontrivial
differential; Table \ref{tab:anss} implies $d_9(\alpha_1\beta_1^2\beta_4) =
\beta_1^8$. Moreover, it is clear from an $E_2(S)$ chart (see
\cite{guozhen-data}) that $\beta_1^8$ cannot be the target of a shorter
differential.
\end{proof}

%\begin{lemma}\label{lem:99}
%The non-$\alpha_1$-divisible generator $x_{99}\in E_2^{99,5}(S)$ is a permanent cycle.
%\end{lemma}
%\begin{proof}
%By \cite[Table A3.4]{green}, $\pi_{99}(S)^\hhat_3$ is 1-dimensional. By Lemma
%\ref{lem:computer}\eqref{item:99}, it suffices to show that $\alpha_1x_{96}$ and
%$\alpha_1\beta_1^7\beta_2$ support nontrivial differentials.
%\fixme{Have to show the targets are nonzero (as well as other place we do this).}
%\end{proof}

\begin{lemma}\label{lem:99,5}
There is a differential $d_5(x_{96})=\beta_1^2x_{75}$.
The $\F_3$-vector space $\alpha_1E_2^{99,5}(S)$ is 1-dimensional and is
generated by a class $\alpha_1x_{99}$ where $x_{99}$ is a permanent cycle.
\end{lemma}
See Lemma \ref{lem:computer} for element definitions. The second sentence is
used implicitly when identifying one of the generators of $E_2^{142,14}(S)$ as
$\alpha_1\beta_1^4x_{99}$ (as seen in Figure \ref{fig:143-S}): Wang's program
only shows that there is a nonzero element in $E_2^{142,14}(S)$ that is
$\alpha_1\beta_1^4$ times an element of $E_2^{99,5}(S)$.

\begin{figure}[h]
\includegraphics[width=160pt]{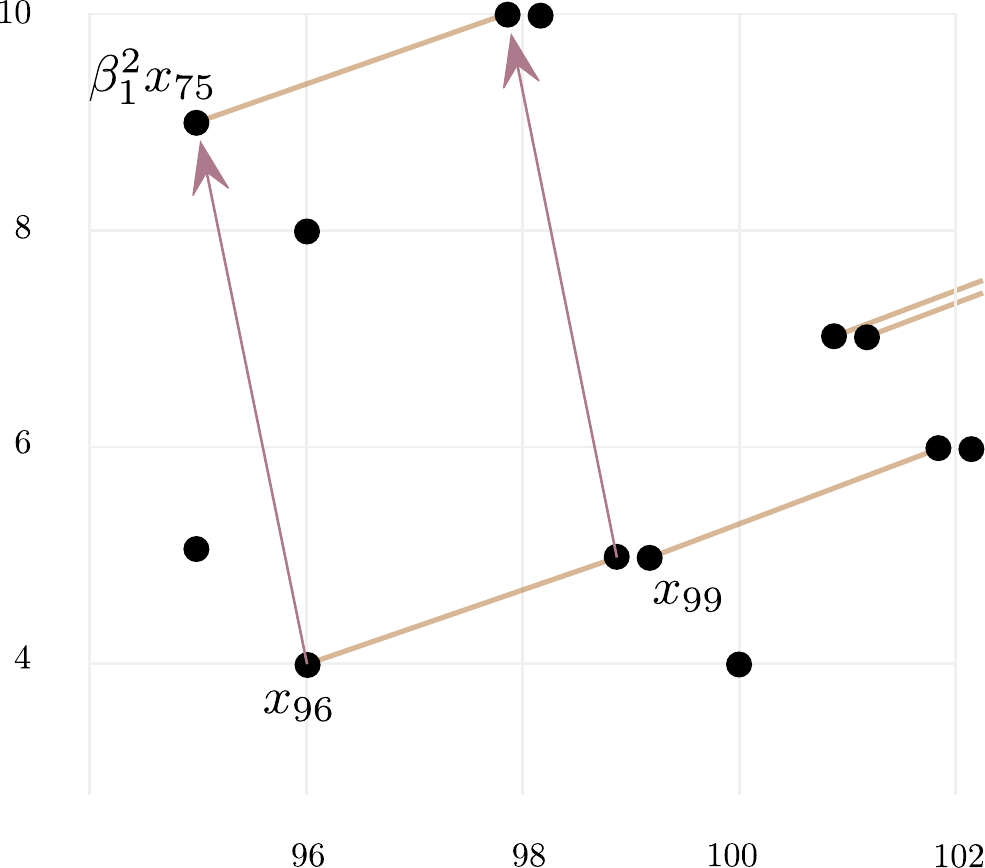}
\caption{$E_2^{s,f}(S)$ in degrees $95\leq s\leq 102$, $4\leq f\leq 10$. Brown
lines represent $\alpha_1$-multiplication. Each dot represents a copy of $\F_3$.
The information in this chart used in the proof of Lemma \ref{lem:99,5} is
summarized in Lemma \ref{lem:computer}\eqref{item:96}\eqref{item:99}.}
\label{fig:18-26}
\end{figure}

\begin{proof}
By \cite[Table A3.4]{green}, $\pi_{96}(S)^\hhat_3=0$, and the generator
$x_{96}\in E_2^{96,4}(S)$ must support a nontrivial differential as it cannot
be a target for degree reasons. We claim this implies a differential
$d_5(x_{96})=\beta_1^2x_{75}$: by
Lemma \ref{lem:computer}\eqref{item:96} the
only other possible target is $\alpha_1\beta_1^4\beta_2^2\in E_2^{95,13}(S)$ (a
possibility for $d_9(x_{96})$), but this is zero in $E_6$ by Lemma
\ref{lem:b6}\eqref{item:ab3} as it is $\alpha_1\beta_1^3$ times the permanent cycle
$\beta_1\beta_2^2$.

From Lemma \ref{lem:computer}\eqref{item:96}, we have $\alpha_1 d_5(x_{96})=
d_5(\alpha_1 x_{96}) = \alpha_1\beta_1^2x_{75}$ is nonzero.
By \cite[Table A3.4]{green}, we have $\pi_{99}(S)^\hhat_3/\Im J \isom \F_3$.
We claim this permanent cycle is detected in filtration 5.
By Lemma \ref{lem:computer}\eqref{item:99}, the only other possibility is
$\alpha_1\beta_1^7 \beta_2\in E_2^{99,17}(S)$, which is the
target of a $d_5$ differential by Lemma \ref{lem:b6}\eqref{item:ab3}. Let $x_{99}$ denote the
permanent cycle in $E_2^{99,5}(S)$. Since $\alpha_1^2=0$ and
$\dim(\alpha_1E_2^{99,5}(S))=1$ by Lemma \ref{lem:computer}\eqref{item:99},
we have that $\alpha_1E_2^{99,5}(S)$ is generated by $\alpha_1x_{99}$.
\end{proof}

\begin{lemma}\label{lem:141,15}
The generator of $E_2^{142,10}(S)$ supports a nontrivial Adams-Novikov $d_5$
differential. The generator of $E_2^{142,6}(S)$ supports a nontrivial
Adams-Novikov $d_9$ differential.
\end{lemma}
\begin{proof}
Combining Lemma \ref{lem:computer}\eqref{item:141} with Lemma \ref{lem:pi81}, we
have that the 2-dimensional vector space $E_2^{141,15}(S)$ is generated by
$\beta_1^6$-divisible permanent cycles. By Lemma \ref{lem:b6}\eqref{item:b6},
both classes in $E_2^{141,15}(S)$ are hit by some differential. 
By the vector space structure of $E_2^{142,*}(S)$ displayed in Figure
\ref{fig:143-S}, the only possibilities are the indicated $d_5$ and $d_9$.
\end{proof}

\begin{lemma}\label{lem:142,14}
The generator of $E_2^{143,9}(S)$ supports a nontrivial Adams-Novikov $d_5$
differential hitting $\alpha_1\beta_1^4 x_{99}$, where $x_{99}$ is the
permanent cycle introduced in Lemma \ref{lem:99,5}. One of the two generators of
$E_2^{143,5}(S)$ supports a nontrivial Adams-Novikov $d_9$ differential hitting
$\beta_1^6 \beta_{6/3}$.
\end{lemma}
\begin{proof}
This proof relies on Figure \ref{fig:143-S}, in
particular the fact that the elements mentioned are all nonzero.
For the first statement, we have $d_5(\beta_{3/3}\cdot \beta_1 x_{99}) =
\alpha_1 \beta_1^3\cdot \beta_1x_{99}$ since $x_{99}$ (and hence
$\beta_1x_{99}$) is a permanent cycle.
Since $\beta_{6/3}\in E_2^{82,2}(S)$ is a permanent cycle by \cite[Table
A3.4]{green}, we may apply Lemma \ref{lem:b6}\eqref{item:b6} to show that $\beta_1^6\beta_{6/3}\in E_2^{142,14}(S)$ is the target of a differential $d_r$ for $r\leq 9$.
Since the group $E_2^{143,9}(S)$ is one-dimensional
and we proved above that the generator supported a nontrivial $d_5$,
$\beta_1^6\beta_{6/3}$ must be hit by a $d_9$.
\end{proof}

%\begin{lemma}\label{lem:143,17}
%There is a nontrivial $d_5$ differential from the generator of $E_2^{143,17}(S)$
%to the generator of $E_2^{142,22}(S)$.
%\end{lemma}
%\begin{proof}
%Let $x_{57}$ denote the generator of $E_2^{57,3}(S)$. Then the generator of
%$E_2^{143,17}(S)$ is $\beta_1^6\beta_2x_{57}$. Using a differential in
%Table~\ref{tab:anss}, we have a differential $d_5(\beta_1^6\beta_2x_{57}) =
%\beta_1^9\beta_2^2$.
%\end{proof}

%\begin{lemma}\label{lem:v1-divisibility}
%~\begin{enumerate} 
%\item No nontrivial elements in $E_2^{142,14}(S/3)$ are $v_1$-divisible.
%\item No nontrivial elements in $E_2^{143,5}(S/3)$ are $v_1^2$-divisible.
%\end{enumerate}
%\end{lemma}
%\begin{proof}
%These statements are both true for degree reasons.
%If $v_1 x \in E_2^{142,14}(S/3)$ then $x$ would lie in $E_2^{138,14}(S/3)$.
%But we have $E_2^{138,14}(S) = E_2^{137,15}(S)=0$, and so
%$E_2^{138,14}(S/3)=0$. Similarly, $E_2^{135,5}(S/3)=0$ since
%$E_2^{135,5}(S)=E_2^{134,6}(S)=0$.
%\end{proof}

\begin{proposition}\label{prop:143-only-5}
Every element in $\pi_{143}(S/3)$ is detected in Adams-Novikov filtration $\leq 5$.
\end{proposition}
\begin{proof}
We list the elements in $E_2^{143,f}(S/3)$ for $f> 5$.

\renewcommand{\figurename}{Table}
\begin{figure}[H]
\begin{tabular}{c|c|c}
Filtration $f$ &  \# bottom cell generators & \# top cell generators
\\\hline\hline 9 & 1 & 1
\\\hline 13 & 0 & 2
\\\hline 17 & 1 & 0
\\\hline 21 & 0 & 1
\\\hline 29 & 1 & 0
\end{tabular}
\caption{Classes in $E_2^{143,f}(S/3)$ for $f\geq 2$}
\end{figure}

We encourage the reader to refer to Figure \ref{fig:143-moore}, which is derived
from Figure \ref{fig:143-S}, alongside the rest of the proof.

\emph{Filtration 9:} We claim that both classes $E_2^{143,9}(S/3)$ support $d_5$ differentials. The
bottom cell class does so because of Lemmas \ref{lem:142,14} and
\ref{lem:d5}\eqref{item:i-transfer}, and the top cell class does so because of Lemma
\ref{lem:141,15} and Lemma \ref{lem:d5}\eqref{item:delta-transfer}.

\emph{Filtration 13:} 
We may take the two generators of $E_2^{143,13}(S/3)$ to be classes
$\bar{\alpha_1\beta_1^4x_{99}}$ and
$\bar{\beta_1^6\beta_{6/3}}$ defined such that their image under $j$ is
$\alpha_1\beta_1^4x_{99}$ and $\beta_1^6\beta_{6/3}$, respectively.
By Lemma \ref{lem:d5}\eqref{item:delta-transfer}, the $d_5$ in Lemma \ref{lem:142,14} induces a
$d_5$ differential hitting
$\bar{\alpha_1\beta_1^4x_{99}}$; note that $\Im(i)=0$ in this degree.

By Lemma \ref{lem:142,14} we have a class $t_2\in E_2^{143,5}(S)$
such that $d_9(t_2) = \beta_1^6\beta_{6/3}$. Let $\bar{t_2}$ be the top cell class in
$E_2^{144,4}(S/3)$ associated to the 3-torsion element $t_2$.
We wish to show that there is a differential $d_9(\bar{t_2}) = \bar{\beta_1^6\beta_{6/3}}$. 
First we check that $\bar{t_2}$ survives to the $E_9$ page.
The only possible targets
for such a shorter differential are in $E_2^{143,9}(S/3)$, and we showed
above that these both support nontrivial $d_5$ differentials.
The map induced by $j$ on $E_2$ pages shows that $d_9(\bar{t_2}) \equiv
\bar{\beta_1^6\beta_{6/3}}$ modulo $\ker(j)$. We have $E_9^{143,13}(S/3) =
\F_3\{ \bar{\beta_1^6\beta_{6/3}} \}$, and
$j(\bar{\beta_1^6\beta_{6/3}}) = \beta_1^6\beta_{6/3}$ which is nonzero in
$E_9(S)$. Thus there is a nonzero $d_9$ differential as claimed.

\emph{Filtration 17:}
The generator of $E_2^{143,17}(S)$ is $\beta_1^6\beta_2x_{57}$, where $x_{57}$
is the generator of $E_2^{57,3}(S)$. Using a differential in
Table~\ref{tab:anss}, we have a differential $d_5(\beta_1^6\beta_2x_{57}) =
\beta_1^9\beta_2^2$. By Lemma~\ref{lem:d5}\eqref{item:i-transfer} we have a differential
$d_5(i(\beta_1^6\beta_2x_{57})) = i(\beta_1^9\beta_2^2)$.

\emph{Filtration 21:}
By Lemma \ref{lem:d5}\eqref{item:delta-transfer}, the $d_5$ differential on $\beta_1^6\beta_2x_{57}$
discussed in the filtration 17 case above gives rise to a differential
$d_5(\bar{\beta_1^6\beta_2x_{57}}) = \bar{\beta_1^9\beta_2^2}$ over $S/3$.

\emph{Filtration 29:} The generator of $E_2^{143,29}(S/3)$ is
$i(\alpha_1\beta_1^{14})$; this class is zero in $E_6(S/3)$ by Lemma
\ref{lem:b6}\eqref{item:ab3}.
\end{proof}

\begin{remark}
The dependence of Proposition \ref{prop:143-only-5} on computer calculations
would be reduced if we could make precise the observation that
much of the Adams-Novikov $E_2$-page is $\beta_1$-periodic, and classes in high
filtrations are highly $\beta_1$-divisible. Using \cite[Theorem 2.3.1, Remark
2.3.5(c)]{palmieri-book}, one can prove that multiplication by $\beta_1$ is an
isomorphism on the Adams $E_2$ page restricted to Adams filtration $f_A$, stem
$s$, and filtration $\nu$ in the algebraic Novikov spectral sequence
$\Ext^{*,*}_A(\F_3,\F_3)\implies E_2^{*,*}(S)$ if
$$ f_A > \tf{1\over 23}s + \tf{24\over 23}\nu + \tf{159\over 23}. $$
By keeping track of the effect on the algebraic Novikov spectral sequence, one
can derive that $\beta_1$ acts injectively (up to higher algebraic Novikov
filtration) on the subspace of $E_2^{s,f}(S)$ in
algebraic Novikov filtration $\nu$ if 
\begin{equation}\label{eq:inj-region} f > \tf{1\over 23}s + \tf{1\over 23}\nu + \tf{169\over 23}.
\end{equation}
Surjectivity is harder to prove.
Even if we knew that $\beta_1$ acted isomorphically on the region
\eqref{eq:inj-region} (which is
often true), this is not enough to prove the $\beta_1$-divisibility results we
need.
For example, in Proposition \ref{prop:143-only-5} we use the fact that the generator
$x$ of $E_2^{143,17}(S)$ is divisible by $\beta_1^6$. This element has $\nu=0$,
and $\beta_1^{-1}x$ and $\beta_1^{-2}x$ lie in the region
\eqref{eq:inj-region} but $\beta_1^{-3}x$ does not.
Improving this bound would also be of use more generally to the study of the
3-primary Adams and Adams-Novikov spectral sequences.
\end{remark}

\renewcommand{\figurename}{Figure}
\begin{minipage}[t]{0.5\textwidth}
\begin{figure}[H]
\vspace{-20pt}
\includegraphics[width=95pt]{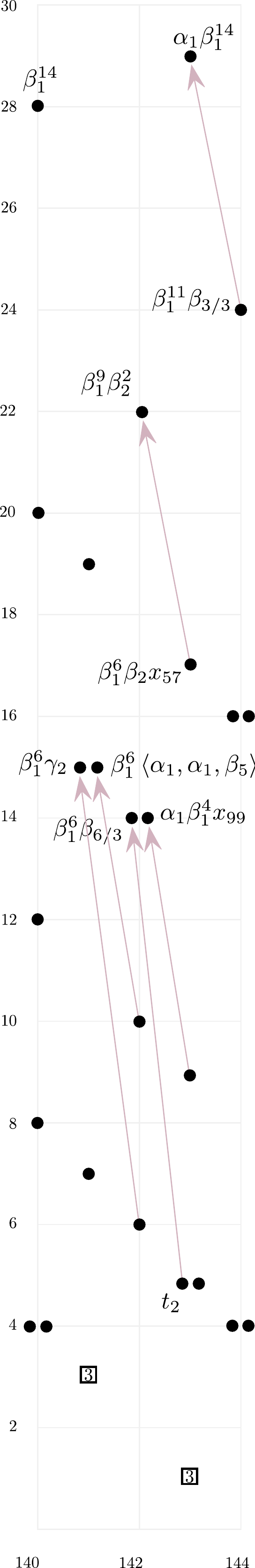}
\vspace{10pt}
\caption{$E_2^{s,f}(S)$ in degrees $140\leq s\leq 144$, along with
some Adams-Novikov differentials. A box containing ``3'' denotes a copy
of $\Z/27$. Multiplications by $\alpha_1$ are not shown.}
\label{fig:143-S}
\end{figure}
\end{minipage}
\begin{minipage}[t]{0.5\textwidth}
\begin{figure}[H]
\vspace{-20pt}
\includegraphics[width=120pt]{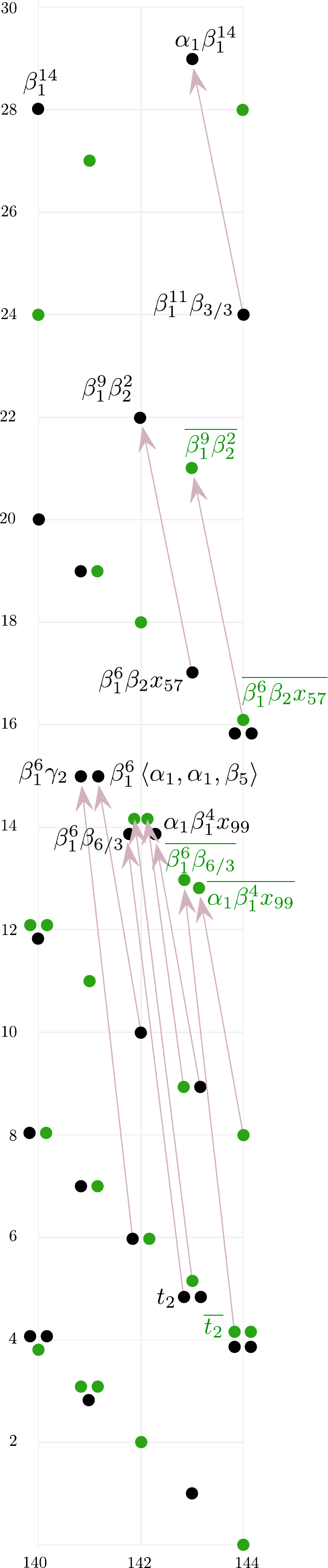}

\vspace{10pt}
\caption{$E_2^{s,f}(S/3)$ in degrees $140\leq s\leq 144$, along with
some Adams-Novikov differentials. Green dots denote top cell classes.
Multiplications by $\alpha_1$ are not shown.}
\label{fig:143-moore}
\end{figure}
\end{minipage}

\section{Survival of $v_2^9$}\label{sec:v2^9}
In this section, we prove Theorem \ref{thm:v2^9}, which says that $v_2^9$ is a
permanent cycle in $E_2(S/(3,v_1^8))$. We first explain the choice of exponent
of $v_1$.
Since $\eta_R(v_2) \equiv v_2 + v_1t_1^3 -v_1^3t_1 \pmod 3$ in the Hopf
algebroid $(BP_*, BP_*BP)$ (see e.g. \cite[(6.4.16)]{green}), we have that
$v_2^3$ is an element of $E_2(S/(3,v_1^m))$ for $m\leq 3$, and $v_2^9$ is an
element of $E_2(S/(3,v_1^m))$ for $m\leq 9$.
On the other hand, we would like to work with $m\geq 8$, since those are the
values of $m$ for which $\beta_{9/8}$ is in the image of the
composition of Adams-Novikov $E_2$ page boundary maps $E_2(S/(3,v_1^m)) \to
E_2(S/3)\to E_2(S)$. 
Trivial modifications to the work in this section show that $v_2^9\pm
v_1^8v_2^7$ is a self-map on $S/(3,v_1^9)$; see Remark \ref{rmk:m=9}. However,
this slight strengthening is not necessary for our purposes, and we write down
our results for $v_2^9\in \pi_*(S/(3,v_1^8))$ essentially for cosmetic reasons,
avoiding the correction term. To obtain the families in
Theorem \ref{thm:beta} other than $\beta_{9t+9/j}$, it suffices to work with
$S/(3,v_1^4)$.

The main ingredients for proving Theorem \ref{thm:v2^9} are Lemma
\ref{lem:143-v1div} and Proposition \ref{prop:143-only-5} from the previous
section, and the following lemma (below, specialized to our setting) due to
the second author. It draws a connection between hidden $v_1^8$-extensions in
$\pi_*(S/3)$, and differentials of the minimum length (i.e., $d_5$
differentials) in the Adams-Novikov spectral sequence for $S/(3,v_1^8)$.

\begin{lemma}[{\cite[Lemma 1.4]{shimomura-L2-toda-smith}}]
\label{lem:extensions-diff}
Let $m\geq 1$.
Suppose we have $y\in E_5(S/(3,v_1^m))$ such that $j_m(y)$ is a
nontrivial permanent cycle in $E_5^{s,f}(S/3)$, and let $w$ denote an element
in $E_5(S/3)$ detecting the product $v_1^m\cdot \{ j_m(y) \} \in \pi_*(S/3)$. Then there is a differential
$$ d_5(y) = i_m(w) $$
in $E_5(S/(3,v_1^m))$.
\end{lemma}

In the next lemma we separate out the general strategy used to prove that
$v_2^9$ and other elements in Section \ref{sec:betas} are permanent
cycles.
\begin{lemma}\label{lem:perm-cycle-utility}
~\begin{enumerate} 
\item\label{item:utility-3} Let $x\in E_2^{s,f}(S/3)$ for $f\leq 3$ be an
element such that $j(x)\in E_2^{s-1,f+1}(S)$ is a permanent cycle and $\{ j(x)
\}\in \pi_{s-1}(S)$ is an essential element of order 3. Furthermore, suppose that
$\Im(i: E_2^{s,f}(S)\to E_2^{s,f}(S/3))$ consists of permanent cycles. Then
$x\in E_2^{s,f}(S/3)$ is a permanent cycle.
\item\label{item:utiliy-v1} Let $x\in E_2^{s,f}(S/(3,v_1^m))$ for $f\leq 3$ be
an element such that $j_m(x)\in E_2^{s-4m-1,f+1}(S/3)$ is a permanent cycle and
$\{ j_m(x) \}$ is an essential element with $v_1^m\cdot \{ j_m(x) \}=0\in \pi_*(S/3)$.
Furthermore, suppose that $\Im(i_m:E_2^{s,f}(S/3)\to E_2^{s,f}(S/(3,v_1^m)))$
consists of permanent cycles. Then $x\in E_2^{s,f}(S/(3,v_1^m))$ is a permanent cycle.
\end{enumerate}
\end{lemma}
\begin{proof}
We just prove (1), as (2) is analogous. Consider the exact
sequences
\begin{align}\label{eq:E2-LES} & E_2^{s,f}(S)\too{3} E_2^{s,f}(S)\too{i} E_2^{s,f}(S/3)\too{j} E_2^{s-1,f+1}(S)
\\\notag  & \pi_s(S)\too{3}
\pi_s(S)\too{i}\pi_s(S/3)\too{j}\pi_{s-1}(S)\too{3}\pi_{s-1}(S)
\end{align}
associated to the cofiber sequence $S\too{3}S\too{i} S/3\too{j} S$. (For the
first long exact sequence, we are using the fact that $j$ induces the zero map in
$BP$-homology.) Suppose
that $x\in E_2^{s,f}(S/3)$ is an element such that $j(x)$ is a permanent cycle
with $3\cdot \{ j(x) \}=0$. Then there exists an element $\xi\in \pi_s(S/3)$ such
that $j(\xi) = \{ j(x) \}$. Since $j:S/3 \to \Sigma S$ induces a map of
Adams-Novikov spectral sequences, the induced map on homotopy $j:\pi_*(S/3)\to \pi_*(\Sigma
S)=\pi_{*-1}(S)$ respects Adams-Novikov filtration; thus $j(x)$ being detected
in filtration $f+1$ implies $\xi$ is detected in filtration
$\leq f$. The assumption $f\leq 3$ combined with Fact \ref{fact:sparseness}
implies that $\xi$ is detected in filtration $f$. We may write the detecting
element as $x+y$ for some $y\in E_2^{s,f}(S/3)$. By the Geometric Boundary
Theorem, $j(x+y)$ converges to $j(\xi)$, and we also have that $j(x)$ converges
to $j(\xi)$. So $j(y)$ is a boundary. But $j(y)$ has filtration $\leq 4$, so
Fact \ref{fact:sparseness} implies $j(y)=0$ in $E_2(S)$. By \eqref{eq:E2-LES},
we have that $y$ is in the image of $i$. By the assumption about $\Im(i)$, $y$
is a permanent cycle, and we have from above that $x+y$ is a permanent cycle.
Therefore, $x$ is a permanent cycle.
\end{proof}

\begin{lemma}\label{lem:delta-perm}
The element $\bar{\beta_{9/8}} =j_8(v_2^9)\in E_2^{111,1}(S/3)$ is a permanent cycle.
\end{lemma}
\begin{proof}
By \cite[Table A3.4]{green}, there exists $c\in \{ \pm
1 \}$ such that $x_{106}=\beta_{9/9} + c\beta_7\in E_2^{106,2}(S)$ is a 3-torsion
permanent cycle. Lemma \ref{lem:perm-cycle-utility}\eqref{item:utility-3}
applies since $E_2^{*,1}(S)$ consists of permanent cycles; it implies that
$\bar{\beta_{9/9}} + c\bar{\beta_7} = j_9(v_2^9) + cj_1(v_2^7)\in
E_2^{107,1}(S/3)$ is a permanent cycle.
Hence $v_1\cdot (j_9(v_2^9)+cj_1(v_2^7)) = v_1\cdot
j_9(v_2^9) = j_8(v_2^9)$ is a permanent cycle.
\end{proof}

\begin{lemma} \label{lem:d5(v2^9)}
We have $d_5(v_2^9) = 0$ in $E_5^{143,5}(S/(3,v_1^8))$.
\end{lemma}
\begin{proof}
Let $x = d_5(v_2^9)$. We first consider the image of this differential along the
natural map induced by $i'_3: S/(3,v_1^8)\to S/(3,v_1^3)$.
Since $v_2^3$ is an element of $E_2^{48,0}(S/(3,v_1^3))$,
by the Leibniz rule and Theorem \ref{thm:oka-ring}, we have $i'_3(x) =
i'_3(d_5(v_2^9)) = d_5((v_2^3)^3) = 3v_2^6d_5(v_2^3) = 0$.

Using Lemma \ref{lem:delta-perm}, we have
$$ j_8(x)=j_8(d_5(v_2^9)) = d_5(j_8(v_2^9)) = 0$$
in $E_5(S/3) = E_2(S/3)$.
Thus the exact sequence
$$ E_2^{143,5}(S/3)\too{i_8} E_2^{143,5}(S/(3,v_1^8))\too{j_8} E_2^{111,1}(S/3)$$
gives $x = i_8(y)$ for some $y\in E_2^{143,5}(S/3)$.

Consider the commutative diagram of cofiber sequences obtained using Verdier's
axiom:
\begin{equation}\label{eq:verdier} \xymatrix{
S/3\ar[r]^-{v_1^3}\ar[d]_-{i_5} & S/3\ar[r]^-{i_3}\ar[d]^-{i_8} &
S/(3,v_1^3)\ar@{=}[d]
\\S/(3,v_1^5)\ar[r]^-{v_1^3}\ar[d]_-{j_5} & S/(3,v_1^8)\ar[r]^-{i'_3}\ar[d]^-{j_8} &
S/(3,v_1^3)\ar[d]
\\S/3\ar@{=}[r] & S/3\ar[r] & \ast.
}\end{equation}
This implies that $i_3 = i'_3\circ i_8$; in particular, $i_3(y) =
i'_3(i_8(y))=i'_3(x)=0$.
By the long exact sequence corresponding to the top row of \eqref{eq:verdier},
we have that $y$ is $v_1^3$-divisible. By Lemma \ref{lem:143-v1div}, $y=0$.
\end{proof}

\begin{lemma}\label{lem:b9/8-extn}
The product $v_1^8\cdot \{j_8(v_2^9)\}$ is zero in $\pi_*(S/3)$.
\end{lemma}
\begin{proof}
Let $w\in E_5^{143,f}(S/3)$ be a representative of
$v_1^8\cdot \{j_8(v_2^9)\}\in \pi_{143}(S/3)$.
By Proposition \ref{prop:143-only-5}, we have $f\leq 5$, and by Fact
\ref{fact:sparseness}, the only possibilities are $f=1,5$. The product
$v_1^8\cdot j_8(v_2^9)$ is zero on the $E_2$ page, so we must have $f=5$.

By Lemma \ref{lem:delta-perm}, Lemma \ref{lem:extensions-diff} applies to
$v_2^9$; combining this with Lemma \ref{lem:d5(v2^9)}, we have
$$ 0 = d_5(v_2^9) = i_8(w) $$
in $E_5^{143,5}(S/(3,v_1^8)) = E_2^{143,5}(S/(3,v_1^8))$. Thus $w$ is divisible
by $v_1^8$ in $E_2^{143,5}(S/3)$. By Lemma \ref{lem:143-v1div}, we must have $w=0$ in
$E_5^{143,5}(S/3)$. Since $w$ was defined to be an element detecting $v_1^8\cdot
\{ j_8(v_2^9) \}$, this product is zero in homotopy.
\end{proof}

\begin{theorem}\label{thm:v2^9}
The element $v_2^9\in E_2^{144,0}(S/(3,v_1^8))$ is a permanent cycle in the Adams-Novikov
spectral sequence computing $\pi_*(S/(3,v_1^8))$.
\end{theorem}
\begin{proof}
This will follow from applying Lemma
\ref{lem:perm-cycle-utility}\eqref{item:utiliy-v1} to $v_2^9$. The first two
hypotheses of that lemma are satisfied due to Lemma \ref{lem:delta-perm} and
Lemma \ref{lem:b9/8-extn}. For the last hypothesis, note that $E_2^{144,0}(S/3)$
is generated by $v_1^{36}$, which is a permanent cycle.
\end{proof}

\begin{corollary}\label{cor:self-map}
For $2\leq m\leq 8$,
the class $v_2^9\in \pi_{144}(S/(3,v_1^m))$ lifts to a class $v_2^9\in
[S/(3,v_1^m), S/(3,v_1^m)]_{144}$.
\end{corollary}
\begin{proof}
Naturality of the map $S/(3,v_1^8)\to S/(3,v_1^m)$ for $m\leq 8$ means that
Theorem \ref{thm:v2^9} directly implies $v_2^9\in E_2^{144,0}(S/(3,v_1^m))$ is a
permanent cycle. By \cite[Theorem 6.1]{oka-ring}, $R=S/(3,v_1^m)$ is a
(homotopy) ring spectrum for $m\geq 2$. Thus the desired self map may be obtained as
\begin{align*}
R & \to S\sm R \too{v_2^9\sm I} R\sm R\too{\mu} R. \qedhere
\end{align*}
\end{proof}

\begin{remark}\label{rmk:m=9}
Essentially the same argument shows that $v_2^9\pm v_1^8v_2^7$
is a permanent cycle in $E_2^{144,0}(S/(3,v_1^9))$. In the proof of Lemma
\ref{lem:delta-perm} we show $j_9(v_2^9)\pm j_1(v_2^7) = j_9(v_2^9\pm
v_1^8v_2^7)$ is a permanent cycle. The proofs of Lemmas \ref{lem:d5(v2^9)} and
\ref{lem:b9/8-extn} go
through without modification to show that $d_5(v_2^9\pm v_1^8v_2^7) = 0$ in
$E_5^{143,5}(S/(3,v_1^9))$ and $v_1^9\cdot \{ j_9(v_2^9\pm v_1^8v_2^7) \}=0$.
In the proof of Theorem \ref{thm:v2^9}, we have $x = c_1(v_2^9\pm v_1^8v_2^7) +
c_2v_2^9 + c_3v_1^{36}$. This time, $v_2^9$ and $v_1^8v_2^7$ are not permanent
cycles since $\beta_{9/9}=j(j_9(v_2^9))$ and $\beta_7=j(j_9(v_1^8v_2^7))$ are
not permanent cycles, and $v_1^{36}$ is a permanent cycle. Thus $v_2^9\pm
v_1^8v_2^7$ is a permanent cycle for some choice of sign.
\end{remark}

\section{Survival of beta elements}\label{sec:betas}
Our goal in this section is to prove that several infinite families of $\beta_{a/b}$
elements are permanent cycles in the Adams-Novikov spectral sequence for the
sphere. For indices $a,b$ satisfying the
conditions in \cite[Theorem 2.6]{MRW}, Miller, Ravenel, and Wilson define cycles
$\beta_{a/b}$ in $E_2^{*,2}(S)$ as the image of certain
classes in $E_2^{*,0}(S/(3,v_1^b))$ under the composition $j\circ j_b$.
In this section we will only consider classes $\beta_{sp^n/b}$ (with $p\nmid s$)
such that $b\leq p^n$, which enables us to use the equivalent, but
simpler, definition
$$ \beta_{a/b} = j(j_b(v_2^a))\in E_2^{16a-4b-2,2}(S) $$
(at $p=3$). These elements are defined using the boundary maps $j$ and $j_b$ on
$\Ext$ associated to the short exact sequences $BP_* \too{3} BP_* \to BP_*/3$
and $BP_*/3 \too{v_1^b} BP_*/3 \to BP_*/(3,v_1^b)$; by the Geometric Boundary
Theorem these coincide with the maps induced on Adams-Novikov spectral sequences
by the maps $j$ and $j_b$ of spectra that we have been considering
in this paper. Recall the convention that $\beta_a := \beta_{a/1}$.

Suppose $\beta_{a/b}$ is a permanent cycle with $b\leq 8$, and in addition, suppose that the
corresponding element in homotopy $\beta^h_{a/b}\in \pi_*(S)$ factors as
\begin{equation}\label{eq:factorization}
\beta^h_{a/b} : S\ttoo{B_{a/b}} S/(3,v_1^b)\ttoo{jj_b} S
\end{equation}
for some $B_{a/b}\in \pi_*(S/(3,v_1^b))$.
In this case,
for $t\geq 1$, Corollary \ref{cor:self-map} allows us to define
elements in $\pi_*(S)$:
$$ \beta^h_{9t+a/b}: S\ttoo{B_{a/b}} S/(3,v_1^b) \ttoo{(v_2^9)^t}
S/(3,v_1^b)\ttoo{jj_b} S. $$
We warn that existence of a factorization \eqref{eq:factorization} is not
automatic, even if $\beta_{a/b}$ is a permanent cycle and such a factorization
exists on the level of Adams-Novikov $E_2$ pages.
%That is, we have elements $\beta_{9t+a/b}$ in the Adams-Novikov $E_2$ page,
%but it takes work to show that they represent homotopy elements, and we have 
%related elements $\beta^h_{9t+a/b}\in \pi_*S$ but it takes work to show
%that they are nonzero.
Our goal is to show the following, proved at the end of the section.

\begin{theorem}\label{thm:beta}
For all $t\geq 0$, the classes
\begin{align*}
\beta_{9t+3/j}  & \text{ for }j=1,2
\\\beta_{9t+6/j} &\text{ for } j=1,2,3
\\\beta_{9t+9/j} & \text{ for } j=1,\dots,8
\\\alpha_1\beta_{9t+3/3} & 
\\\alpha_1\beta_{9t+7} & 
\end{align*}
are permanent cycles in
the Adams-Novikov spectral sequence for the sphere.
\end{theorem}

Since $\beta_{3/3}$ and $\beta_7$ support Adams-Novikov differentials, none
of the families in
Theorem \ref{thm:beta} are trivially multiplicative consequences of a
different family. Instead, we have $\alpha_1\beta_{3/3}\in
\an{\alpha_1,\alpha_1,\beta_1^3}$ and $\alpha_1\beta_7\in
\an{\alpha_1,\alpha_1,\beta_1^2\beta_{6/3}}$.
As we will see in Section \ref{sec:tmf}, the families $\alpha_1\beta_{9t+3/3}$,
$\beta_{9t+6/3}$, and $\alpha_1\beta_{9t+7}$ have nontrivial image in
$\pi_*\tmf$, along with the family $\beta_{9t+1}$ constructed in
\cite[Corollary 1.2]{behrens-pemmaraju}.

\begin{lemma}\label{lem:beta32bar*v1}
The class $j_4(v_1^2v_2^3)\in E_2^{39,1}(S/3)$ is a permanent cycle such that
$j(j_4(v_1^2v_2^3)) = \beta_{3/2}\in E_2^{38,2}(S)$ and
$v_1^4\cdot \{j_4(v_1^2v_2^3)\}=0$ in $\pi_*(S/3)$.
\end{lemma}
\begin{proof}
We have $j(j_4(v_1^2v_2^3)) = j(j_2(v_2^3)) = \beta_{3/2}$ in $E_2(S)$ by the
definition of the $\beta$ elements along with Lemma \ref{lem:delta-switch}. By classical computations of the
Adams-Novikov $E_2$ page (see e.g. \cite[Figure 1.2.19]{green}),
$E_2^{38,f}(S)=0 = E_2^{37,f+1}(S)$ for $f\geq 3$, so $E_2^{38,f}(S/3)=0$ for
$f\geq 3$. Thus $j_4(v_1^2v_2^3)\in E_2^{39,1}(S/3)$
cannot support a differential of any length.

As $v_1^4\cdot j_4= 0$ as a map $E_2(S/(3,v_1^4))\to E_2(S/3)$,
it remains to rule out hidden $v_1^4$-extensions on $\beta'_{3/2}:= \{
j_4(v_1^2v_2^3) \}\in \pi_{39}(S/3)$.
Using \cite[Table A3.4]{green} we have $\pi_{51}(S/3) = \F_3\{ \alpha_{13},
\bar{\beta_1^5} \}$ and so $v_1^3\cdot \beta'_{3/2} =
c\bar{\beta_1^5} = c\bar{\beta_1^2}\cdot \beta_1^3$ for some $c\in \F_3$.
(If there were an $\alpha_{13}$ component, then the extension would not be
hidden.)
We have $v_1\cdot \bar{\beta_1^2}=0$ for degree reasons, as Ravenel's table implies
$\pi_{25}(S/3)=0$. Thus $v_1\cdot v_1^3\cdot \beta'_{3/2}=0$.
\end{proof}

\begin{lemma}\label{lem:6/3}
The class $j_4(v_1v_2^6)\in E_2^{83,1}(S/3)$ is a permanent cycle such that
$j(j_4(v_1v_2^6)) = \beta_{6/3}\in E_2^{82,2}(S)$ and
$v_1^4\cdot \{j_4(v_1v_2^6)\}=0$ in $\pi_*(S/3)$.
\end{lemma}

\begin{proof}
By Ravenel's table \cite[Table A3.4]{green}, $\beta_{6/3}$ and $\beta_6$ are
3-torsion permanent cycles. Since $j(j_4(v_1v_2^6)) = \beta_{6/3}$ and
$j(j_4(v_1^3v_2^6)) = \beta_6$, we apply Lemma
\ref{lem:perm-cycle-utility}\eqref{item:utility-3} to $j_4(v_1v_2^6)\in
E_2^{83,1}(S/3)$ and $j_4(v_1^3v_2^6)\in E_2^{91,1}(S/3)$, noting that
$E_2^{*,1}(S)$ consists of permanent cycles. This shows that $j_4(v_1v_2^6)$ and
$j_4(v_1^3v_2^6)$ are permanent cycles.

To determine $v_1^4\cdot \{j_4(v_1v_2^6)\}$, we first consider the
possibilities for $v_1^2\cdot \{j_4(v_1v_2^6)\}\in \pi_{91}(S/3)$: from Ravenel's tables, we have
$$ \pi_{91}(S/3) = \F_3\{ \alpha_{23}, \beta_1\gamma_2, \beta_1x_{81},
\beta'_6 \} $$
where $j(\beta'_6)=\beta_6$.
Since $v_1^2\cdot \beta_1 = 0$ in homotopy by Lemma
\ref{lem:b6}\eqref{item:v1^2beta}, we have $v_1^2\cdot
\{j_4(v_1v_2^6)\}=c_1\alpha_{23} + c_2\beta'_6$ for $c_i\in \F_3$. If $c_1\neq 0$ then
$v_1^4\cdot \{j_4(v_1v_2^6)\}$ would be detected in filtration 1,
contradicting the fact that $v_1^4\cdot j_4(v_1v_2^6)=0$ in $E_2(S/3)$. So
it suffices to show that $v_1^2\cdot \beta'_6 = 0$.
From above, we may write $\beta'_6 = \{j_4(v_1^3v_2^6)\} = \{ j_2(v_1v_2^6) \}$.
By \cite[Lemma 3]{oka-beta}, $v_1v_2^6\in E_2^{100,0}(S/(3,v_1^2))$ is a
permanent cycle, and hence so is $j_2(v_1v_2^6)$. We have
$\{ j_2(v_1v_2^6) \} = j_2(\{v_1v_2^6\})$ by the Geometric Boundary theorem, and $v_1^2\cdot
j_2(\{v_1v_2^6\})=0$ by definition of $j_2$ as a map on homotopy groups.
\end{proof}

\begin{lemma}\label{lem:v1^2v2^3}
The classes $v_1^2v_2^3\in E_2^{56,0}(S/(3,v_1^4))$ and
$v_1v_2^6\in E_2^{100,0}(S/(3,v_1^4))$ are permanent
cycles in the Adams-Novikov spectral sequence computing $\pi_*(S/(3,v_1^4))$.
\end{lemma}
\begin{proof}
Use Lemma \ref{lem:perm-cycle-utility}\eqref{item:utiliy-v1}, with
Lemmas \ref{lem:beta32bar*v1} and \ref{lem:6/3} as input. To check the condition
about the image of $i_4:E_2(S/3)\to E_2(S/(3,v_1^4))$ in these degrees, note that
$E_2^{*,1}(S/3)$ is generated by the image of $i:E_2^{*,1}(S)\to
E_2^{*,1}(S/3)$, which consists of permanent cycles (these are all image of $J$
classes), along with elements that
map to $\beta$ elements under $j$. Standard theory about the $\beta$ elements
(\cite{MRW}) implies that $\bar{\beta_{3/2}}$ and $\bar{\beta_{6/3}}$ are the only
such elements in the relevant degrees; these are both permanent cycles by Lemmas
\ref{lem:beta32bar*v1} and \ref{lem:6/3}.
\end{proof}

%Next we turn to the families $\alpha_1\beta_{9t+3/3}$ and $\alpha_1\beta_{9t+7}$.

\begin{lemma}\label{lem:av2^3}
The class $\alpha_1v_1v_2^3\in E_2^{55,1}(S/(3,v_1^4))$ is a permanent cycle in
the Adams-Novikov spectral sequence.
\end{lemma}
\begin{proof}
There is a Toda bracket $\an{\alpha_1,\alpha_1,\beta_1^3}\in \pi_{37}(S)$
detected by $\alpha_1\beta_{3/3}$ in filtration 3, and this class is 3-torsion
(see \cite[Table A3.4]{green}). In order to apply Lemma
\ref{lem:perm-cycle-utility}\eqref{item:utility-3} to
$j_4(\alpha_1v_1v_2^3)\in E_2^{38,2}(S/3)$, we must check that $E_2^{38,2}(S)$
consists of permanent cycles. It follows from standard facts about the
Adams-Novikov 2-line (\cite{MRW}) that $E_2^{38,2}(S)=\F_3\{ \beta_{3/2} \}$. So
we may conclude that $j_4(\alpha_1v_1v_2^3)$ is a permanent cycle.

Moreover, $v_1^4\cdot \{ j_4(\alpha_1v_1v_2^3) \}$ is zero in homotopy: since
$\pi_{53}(S)=0=\pi_{54}(S)$ by \cite[Table A3.4]{green}, we have
$\pi_{54}(S/3)=0$. In order to apply Lemma
\ref{lem:perm-cycle-utility}\eqref{item:utiliy-v1} to $\alpha_1v_1v_2^3\in
E_2^{55,1}(S/(3,v_1^4))$, we must check that $E_2^{55,1}(S/3)$ consists of
permanent cycles. This is true because the image of $E_2^{*,1}(S)$ consists of
permanent cycles, and analysis of the 2-line reveals that there cannot be a
class with nontrivial image in $E_2^{54,2}(S)$. Thus we have that
$\alpha_1v_1v_2^3$ is a permanent cycle.
\end{proof}

\begin{lemma}\label{lem:av2^7-new}
The class $\alpha_1v_1 v_2^7\in E_2^{119,1}(S/(3,v_1^2))$ is a permanent cycle in the
Adams-Novikov spectral sequence.
\end{lemma}
\begin{proof}
Since $S/(3,v_1^2)$ is a ring spectrum (Theorem
\ref{thm:oka-ring}), we may consider this element as a product $v_1v_2^5 \cdot \alpha_1v_2^2$.
Oka \cite[Lemma 2]{oka-beta} showed that $v_2^5$ is a permanent cycle in
$E_2^{80,0}(S/(3,v_1))$. This implies that $v_1v_2^5$ is a permanent cycle in
$E_2^{84,0}(S/(3,v_1^2))$.

Next we consider possible differentials on $\alpha_1v_2^2\in
E_2^{35,1}(S/(3,v_1^2))$.
An element in $E_2^{34,f}(S/(3,v_1^2))$ either has nonzero image under $j_2$ in
$E_2^{25,f+1}(S/3)$ or is the image under $i_2$ of an element of
$E_2^{34,f}(S/3)$.
From classically known computations of
the Adams-Novikov $E_2$ page (e.g. see \cite[Figure 1.2.19]{green}), we deduce
that $E_2^{25,*}(S/3) = 0$ and
$E_2^{35,\geq 3}(S/3)=\F_3\{ \bar{\alpha_1\beta_1^3} \}$. This implies
$E_2^{34,\geq 3}(S/(3,v_1^2))$ is generated by
$i_2(\bar{\alpha_1\beta_1^3})$. Observe that $\bar{\alpha_1\beta_1^3} =
\bar{\alpha_1}\beta_1^3 = v_1\cdot \beta_1^3$.
Thus the only possible nonzero differential on $v_1v_2^5\cdot \alpha_1v_2^2$ is
a $d_5$ with target $v_1v_2^5\cdot v_1\beta_1^3$. But the target is divisible by
$v_1^2$, hence zero in $E_5(S/(3,v_1^2))$.
\end{proof}

\begin{proof}[Proof of Theorem \ref{thm:beta}]
We show that $\beta_{9t+3/2}$ and $\beta_{9t+3/1}$ are permanent cycles for $t\geq 0$. Since $v_2^9$
is a permanent cycle in $E_2(S/(3,v_1^8))$ by Theorem \ref{thm:v2^9}, its image
in $E_2(S/(3,v_1^4))$ is a permanent cycle. Lemma
\ref{lem:v1^2v2^3} says that $v_1^2v_2^3$ is a permanent cycle in
$E_2(S/(3,v_1^4))$, so the product $v_1^2 v_2^3 \cdot v_2^{9t}\in
E_2(S/(3,v_1^4))$ is a permanent cycle. Recall that $\beta_{9t+3/2}\in E_2(S)$ is defined
as $j(j_2(v_2^{9t+3}))= j(j_4(v_1^2v_2^{9t+3}))$ in $E_2(S/3)$. Since
$j_4(v_1^2v_2^{9t+3})$ is a permanent cycle, so is $j(j_4(v_1^2v_2^{9t+3}))$.
Since $v_1^2v_2^3$ is a permanent cycle in $E_2(S/(3,v_1^4))$, so is
$v_1^3v_2^3$, so $\beta_{9t+3/1} = j(j_1(v_2^{9t+3}))= j(j_4(v_1^3v_2^{9t+3}))$
is a permanent cycle in $E_2(S)$.

The family $\beta_{9t+9/8}$ (and hence $\beta_{9t+9/j}$ for $j < 8$) follows directly
from the fact that $v_2^9$ is a permanent cycle in $E_2(S/(3,v_1^8))$.
The other families of permanent cycles follow analogously, using Lemma
\ref{lem:v1^2v2^3} again as the input for $\beta_{9t+6/3} = j(j_4(v_1v_2^{9t+6}))$, Lemma
\ref{lem:av2^3} as the input for $\alpha_1\beta_{9t+3/3} =
j(j_4(\alpha_1v_1v_2^{9t+3}))$, and Lemma
\ref{lem:av2^7-new} as the input for $\alpha_1\beta_{9t+7} =
j(j_2(\alpha_1v_1v_2^{9t+7}))$.
\end{proof}

\section{3-primary Hurewicz image of $\tmf$}\label{sec:tmf}
In this section we determine the image of the Hurewicz map $h:\pi_*S \to
\pi_*\tmf$ induced by the unit map $S\to \tmf$. The target $\pi_*\tmf$ has been
computed via the elliptic spectral sequence (see \cite[\S3]{bauer-tmf}); this is
the $Y(4)$-based Adams spectral sequence for $\tmf$, where $Y(4)$ is the Thom
spectrum of $\Omega U(4)\to \Z\x BU$. We will denote this spectral sequence
by $\Eel_r(\tmf)$.

\begin{theorem}[{Hopkins-Mahowald, Bauer \cite[\S6]{bauer-tmf}}]
\label{thm:homotopy-tmf}
At $p=3$, $\pi_*\tmf$ 
is generated by $c_4$, $c_6$, $\Delta$, $\alpha$, $\beta$, and $b$, subject to
the relation $c_4^3-c_6^2 = 1728\Delta$ and the relations on the other
generators displayed in Figure \ref{fig:tmf}. Multiplication by $\Delta^3\in
\pi_{72}(\tmf)$ is injective.
\end{theorem}

\begin{figure}[H]

\hspace{-5pt}\includegraphics[width=440pt]{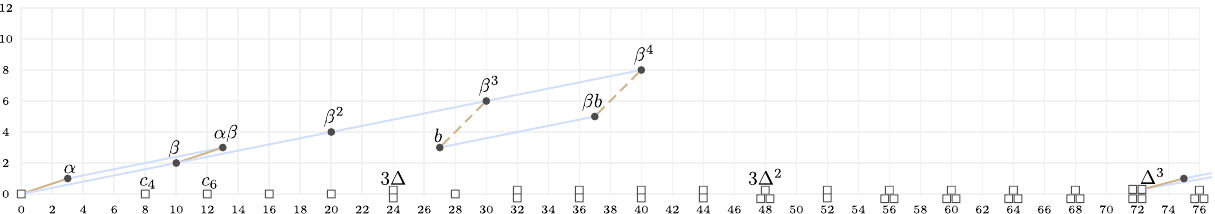}

\caption{The $E_\iy$ page of the elliptic spectral sequence computing
$\pi_s\tmf$ for $0\leq s\leq 76$. Dashed brown lines represent hidden
$\alpha$-multiples. Squares indicate copies of $\Z_{(3)}$ and dots indicate
copies of $\F_3$.}
\label{fig:tmf}
\end{figure}

We will show (Theorem \ref{thm:hurewicz-image}) that all classes in filtration
$\geq 2$ are in the Hurewicz image, and the only classes in filtrations 0 and 1
in the image are the summands generated by $1$ and $\alpha$. Instead of
directly mapping to the elliptic spectral sequence,
we use the $K(2)$-local $E$-based Adams spectral sequence
$$ \EE_2(\TMF) = H^*(G_{24}; E_*) \implies \pi_*(L_{K(2)}\TMF) $$
where $E = E_2$ is height 2 Morava $E$-theory and $\TMF$ is the periodic version
of $\tmf$. There is a map of spectral sequences $E_r(S)\to \EE_r(\TMF)$ induced
by the natural maps $BP\to E$ and $S\to \TMF$. Henn-Karamanov-Mahowald
\cite[Theorem 1.1]{HKM} completely determine $\EE_2(\TMF/3)$ and provide formulas
that we use to compute the map on $E_2$ pages $E_2(S)\to \EE_2(\TMF)$ in cases
of interest (see Lemmas \ref{lem:v2-Delta} and \ref{lem:H(v23)}). For each class in $\EE_2(\TMF)$ in filtration $\geq 2$, we identify
a preimage in $E_2(S)$ that is among the classes proved to be permanent cycles
in Theorem \ref{thm:beta} or \cite{behrens-pemmaraju} (see Proposition
\ref{prop:beta9t+3}).
As we explain in the proof of Theorem \ref{thm:hurewicz-image}, it
suffices to understand the Hurewicz image in $\pi_*(L_{K(2)}\TMF)$ because there
is an injection $\pi_*(\tmf)\to \pi_*(L_{K(2)}\TMF)$ (see Lemma
\ref{lem:tmf-localization-injective}).

First we review some notation and basic facts.
We have $E_*/3 = \F_9[[u_1]][u^{\pm 1}]$, and there is a natural map $BP_*\to
E_*$ that sends $v_1\mapsto u_1u^{-2}$, $v_2\mapsto u^{-8}$, and $v_i\mapsto 0$
for $i>2$. Abusing notation, we will let $v_i$ denote its image in $E_*/3$.

Recall $j:S/3\to \Sigma S$ denotes the boundary map in the cofiber sequence
$S\too{3}S\to S/3$. We will also use $j$ to refer to the map $j\sm
\TMF: \TMF/3 \to \Sigma \TMF$. Similarly, $j_m$ will denote both
boundary maps $S/(3,v_1^m)\to S/3$ and $\TMF/(3,v_1^m)\to \TMF/3$ depending on
context.

\begin{lemma}\label{lem:v2-Delta}
In $E_*$ we have
\begin{align*}
v_2^3 & \equiv -\Delta^2 - v_1^2v_2\Delta \pmod {(3,v_1^6)}
\\v_2^6 & \equiv \Delta^4-v_1^2v_2\Delta^3\pmod {(3,v_1^3)}
%\\v_2^9  & \equiv -\Delta^6 - v_1^6v_2^3\Delta^3 \pmod {(3,v_1^{18})}.
\\v_2^{3^n} & \equiv -\Delta^{2\cdot 3^{n-1}}- v_1^{2\cdot 3^{n-1}}
v_2^{3^{n-1}}\Delta^{3^{n-1}}\pmod {(3,v_1^{2\cdot 3^n})}.
\end{align*}
\end{lemma}
\begin{proof}
The formula $\Delta\equiv (1-\omega^2u_1^2+u_1^4)\omega^2u^{-12}\pmod
{(3,u_1^6)}$ from \cite[Proposition 5.1]{HKM} implies
\begin{align*}
\Delta^2 & \equiv (1-2\omega^2u_1^2+u_1^4)(-v_2^3)\pmod {(3,v_1^6)}
\\v_1^2v_2\Delta  & \equiv v_2^3(\omega^2u_1^2+u_1^4)\pmod {(3,v_1^6)}.
\end{align*}
where $\omega$ denotes an $8^{th}$ root of unity in $\F_9$.
Combining these facts, we obtain the formula for $v_2^3$; the formulas for $v_2^6$ and
$v_2^{3^n}$ follow from it by squaring and successive cubing, respectively.
\end{proof}

Let 
\begin{align*}
H & :E_2(S)\to \EE_2(\TMF)
\\H' & :E_2(S/3)\to \EE_2(\TMF/3)
\\H'_m & :E_2(S/(3,v_1^m))\to \EE_2(\TMF/(3,v_1^m))
\end{align*}
denote the natural maps of spectral sequences. 

\begin{lemma}\label{lem:H(v23)}
We have
\begin{align*}
H(\alpha_1) & =\alpha
& H'(j_3(v_2^3)) & \doteq \Delta\til{\alpha} 
&H'(j_3(v_1^2v_2^7)) & \doteq \Delta^4 \til{\alpha}
\\H(\beta_1) & \doteq\beta
& H(\beta_{3/3}) & \doteq \Delta \beta
& H(\beta_7) & \doteq \Delta^4\beta
\\H'(j_3(v_1^2v_2)) & \doteq \til{\alpha}
&H'(j_3(v_2^6)) & \doteq \Delta^3\til{\alpha}
%&H'(j_9(v_2^9)) & = 0
\\&
& H(\beta_{6/3}) & \doteq \Delta^3\beta
%& H(\beta_{9/9}) & = 0
\end{align*}
where $j(\til{\alpha}) = \beta$. (Here $\doteq$ denotes equality up to
multiplication by a unit.)
\end{lemma}

\begin{proof}
Following Bauer \cite[\S6]{bauer-tmf}, we have $H(\alpha_1)=\alpha$ since they
both come from the cobar class $[t_1]$, and $H(\beta_1) = \beta$ because of the
Massey products $\beta_1 = \an{\alpha_1,\alpha_1,\alpha_1}$ and $\beta =
\an{\alpha,\alpha,\alpha}$. We have $j(\til{\alpha})=\beta$ and
$j(j_3(v_1^2v_2)) = \beta_1$, so $j(H'(j_3(v_1^2v_2)))=
H(j(j_3(v_1^2v_2)))= \beta$. This specifies
$H'(j_3(v_1^2v_2))$ up to the image of $\EE_2(\TMF)$, but
$\EE_2(\TMF/3)$ is 1-dimensional in the degree of $\til{\alpha}$, so there is no
ambiguity.

For the next column, we have in $\EE_2(\TMF/(3,v_1^3))$ that
\begin{align*}
H'(j_3(v_2^3)) & = j_3(H'_3(v_2^3)) = j_3(-\Delta^2 -v_1^2v_2\Delta)
=-\Delta j_3(v_1^2v_2) = -\Delta \cdot \til{\alpha}
\end{align*}
using Lemma \ref{lem:v2-Delta} and the earlier fact about
$H'(j_3(v_1^2v_2))$. Note that $j_3(\Delta^n)=0$ since $\Delta^n$ is in the image of $\EE_2(\TMF/3)$. Now apply $j$ to get the statement about $H(\beta_{3/3})$.
The remaining facts in this column are analogous, using the fact that
$\beta_{6/3} = j(j_3(v_2^6))$.
The last column is also proved similarly, using the fact that $\beta_7 =
j(j_3(v_1^2v_2^7))$. 
\end{proof}

By our convention about naming elements in the image of
the map $BP_*\to E_*$, we have $H'_3(v_2)=v_2$.

\begin{proposition}\label{prop:beta9t+3}
For $t\geq 0$ the map $H:E_2(S)\to \EE_2(\TMF)$ satisfies:
\begin{enumerate} 
\item $H(\beta_{9t+1}) \doteq \Delta^{6t}\beta$
\item $H(\beta_{9t+3/3}) \doteq \Delta^{6t+1}\beta$
\item $H(\beta_{9t+6/3}) \doteq \Delta^{6t+3}\beta$
\item $H(\beta_{9t+7}) \doteq \Delta^{6t+4}\beta$.
\end{enumerate}
\end{proposition}
\begin{proof}
These statements are all proved the same way; we show (2).
First observe that Lemma \ref{lem:v2-Delta} implies $v_2^9\equiv -\Delta^6\pmod
{(3,v_1^6)}$. Using Lemma \ref{lem:v2-Delta} and Lemma \ref{lem:H(v23)} we have:
\begin{align*}
H(\beta_{9t+3/3}) & = H(j(j_3(v_2^{9t+3}))) = j(j_3(H'_3(v_2^{9t+3})))
\\ & = j(j_3(H'_3(v_2^3)\cdot H'_3(v_2^{9t}))) =
j(j_3((-\Delta^2-v_1^2v_2\Delta)\cdot (-1)^t\Delta^{6t}))
\\ & = j(j_3((-1)^{t+1}\Delta^{6t+2})) +
j(j_3((-1)^{t+1}v_1^2v_2\Delta^{6t+1}))
\\ & = 0 + (-1)^{t+1}\Delta^{6t+1}j(j_3(v_1^2v_2)) =
(-1)^{t+1}\Delta^{6t+1}\beta.
\end{align*}
In the last line we are using the fact that $j_3(v_1^2v_2) =
\til{\alpha}$ in $E_2^E(\TMF/3)$ from Lemma \ref{lem:H(v23)}.
\end{proof}

In the next theorem, we show that every element in $\pi_*\tmf$ detected in
filtration $\geq 2$ is in the Hurewicz image. This result is stated without
proof in \cite[\S1]{tmf-book-henriques}, but we do not know of any prior proof
in the literature.

\begin{theorem}\label{thm:hurewicz-image}
The image of the map $h:\pi_*S\to \pi_*\tmf$ at $p=3$ consists of
the $\Z_{(3)}$ summand generated by 1 and the $\F_3$ summands generated by
\begin{align*}
\alpha,\  \Delta^{3t}\beta^i\text{ for } 1\leq i\leq 4,\
\Delta^{3t}\alpha\beta,\  \Delta^{3t}\beta b
\end{align*}
for $t\geq 0$. More precisely, we have
\begin{align*}
h(\alpha_1) & = \alpha
\\h(\beta_1^{i-1}\beta_{9t+1}) & = \Delta^{6t}\beta^i \text{ for }1\leq i\leq 4
\\h(\alpha_1\beta_{9t+3/3}) & = \Delta^{6t}\beta b
\\h(\beta_1^{i-1}\beta_{9t+6/3}) & = \Delta^{6t+3}\beta^i \text{ for }1\leq i\leq 4
\\h(\alpha_1\beta_{9t+7}) & = \Delta^{6t+3}\beta b.
\end{align*}
\end{theorem}
\begin{proof}
Let $\Eel_2(\tmf)$ denote the elliptic spectral sequence for $\tmf$ (see
\cite[\S6]{bauer-tmf}); recall this is the $Y(4)$-based Adams spectral sequence
for $\tmf$. There is a map of spectral sequences $L:\Eel_r(\tmf)\to
\EE_r(\TMF)$ that comes from the map on Adams towers induced by the maps
$Y(4)\to MU_P\to E$ (where $MU_P$ denotes periodic $MU$) and $\tmf\to \TMF$.
These maps assemble into a diagram of spectral sequences as follows.
\begin{equation}\label{eq:tmf-vs-TMF} \xymatrix{
\Eel_2(\tmf)\ar@{=>}[d]\ar[r]^-L & \EE_2(\TMF)\ar@{=>}[d] &
E_2(S)\ar[l]_-H\ar@{=>}[d]
\\\pi_*\tmf\ar[r]^-L & \pi_*L_{K(2)}\TMF & \pi_*S\ar[l]_-H\ar@/^1pc/[ll]^h
}\end{equation}
No element $x\in {\Eel_\iy}^{s,0}(\tmf)$ for $s\neq 0$ is in
the image of $h$: Lemma \ref{lem:tmf-localization-injective}(2) implies
$x$ would be detected in filtration 0 of $\EE_\iy(\TMF)$, and $H:E_\iy(S)\to
\EE_\iy(\TMF)$ is zero in filtration 0 for nonzero stems.

Next we turn to elements detected in filtration 1. We have
$H(\alpha_1)=\alpha$ by Lemma \ref{lem:H(v23)}; since we have $H = L\circ h$ as maps
$\pi_*S\to \pi_*L_{K(2)}\TMF$ and $L$ is injective by Lemma
\ref{lem:tmf-localization-injective}, this implies
$h(\alpha_1)=\alpha\in \pi_*\tmf$. The other elements of $\Eel_\iy(\tmf)$ in
filtration 1 are $\Delta^{3t}\alpha$ for $t\geq 1$ and $\Delta^{3t}b$ for $t\geq
0$; we will show that the permanent cycles they represent are not in the
Hurewicz image. By Lemma \ref{lem:tmf-localization-injective}(1) they are in the
image of $h$ if and only if their images in $\pi_*L_{K(2)}\TMF$ are in the image
of $H$. By Lemma \ref{lem:tmf-localization-injective}(2) they
are also detected in filtration 1 in $\EE_\iy(\TMF)$, so if they were in the
image of $H$, they would be the image of a class in $E_2(S)$ in filtration 0 or
1. We have $E_2^{s,0}(S)=0$ for $s>1$, so it suffices to show that the elements
in $E_2^{s,1}(S)$ except for $\alpha_1$ are in the kernel of $h$.
If $x\in E_2^{s,1}(S)$ with $s>3$ then $i(x) = \alpha_1v_1^k$ for some
$k\geq 1$. 
If $h'$ denotes the map $\pi_*(S/3)\to \pi_*(\tmf/3)$ induced by $h$,
we have $h'(\alpha_1v_1) = 0$ since $\pi_7(\tmf/3)=0$.
Thus $i(h(x)) = h'(i(x))=0$ in $\pi_*(S/3)$, which implies that $h(x)$ is
3-divisible. But Figure \ref{fig:tmf} shows that there are no 3-divisible
nonzero targets in Adams-Novikov filtration 1.

We will now show how to use Proposition \ref{prop:beta9t+3} to derive the
remaining claims about $h$; for
multiplicative reasons, it suffices to show $i=1$ in those statements.
We will illustrate
this with the element $\alpha_1\beta_{9t+3/3}$; the other elements are
analogous, using Theorem \ref{thm:beta} for $\beta_{9t+6/3}$ or
\cite[Corollary 1.2]{behrens-pemmaraju} for $\beta_{9t+1}$ in place of Theorem
\ref{thm:beta} below as necessary.
By Proposition \ref{prop:beta9t+3}, $H(\alpha_1\beta_{9t+3/3}) =
\Delta^{6t+1}\alpha\beta$ in $\EE_2(\TMF)$.
Since $\Delta^{6t+1}\alpha\beta$ is a permanent cycle in $\Eel_2(\tmf)$ converging
to $\Delta^{6t}\beta b$, we have that $\Delta^{6t+1}\alpha\beta$ is a permanent cycle in
$\EE_2(\TMF)$ converging to $\Delta^{6t}\beta b$. Theorem \ref{thm:beta}
shows that $\alpha_1\beta_{9t+3/3}$ is a permanent cycle in the Adams-Novikov
spectral sequence; write $\alpha_1\beta_{9t+3/3}$ for the
(non-$\alpha_1$-divisible) element in homotopy it converges to.
The following diagram summarizes these statements by illustrating \eqref{eq:tmf-vs-TMF} applied to these elements.
$$ \xymatrix{
\Delta^{6t+1}\alpha\beta\ar@{|->}[r]^-L\ar@{|.>}[d] &
\Delta^{6t+1}\alpha\beta\ar@{|.>}[d] &
\alpha_1\beta_{9t+3/3}\ar@{|.>}[d]\ar@{|->}[l]_-H
\\\Delta^{6t}\beta b\ar@{|->}[r]^-L & \Delta^{6t}\beta b &
\alpha_1\beta_{9t+3/3}\ar@{|->}[l]_-H\ar@/^1pc/[ll]^h
}$$
Thus $H:\pi_*S\to \pi_*L_{K(2)}\TMF$ satisfies
$H(\alpha_1\beta_{9t+3/3}) = \Delta^{6t}\beta b$. Since $H$ factors through $h$
and $L:\pi_*\tmf \to \pi_*L_{K(2)}\TMF$ is injective by Lemma
\ref{lem:tmf-localization-injective}(1), we have that $h(\alpha_1\beta_{9t+3/3}) = \Delta^{6t} \beta b$.
\end{proof}

\begin{lemma}\label{lem:tmf-localization-injective}
~\begin{enumerate} 
\item The map $L:\pi_*\tmf\to \pi_*L_{K(2)}\TMF$ is injective on the classes in Theorem
\ref{thm:hurewicz-image}.
\item The map $L: \Eel_\iy(\tmf)\to \EE_\iy(\TMF)$ is injective in filtration
0 and 1.
\end{enumerate}
\end{lemma}
In fact, $L$ is injective on $E_\iy$ pages in all filtrations, but we do not
need this fact.
\begin{proof}
(1) We have
$$ \pi_*(L_{K(2)}\TMF) = \left(\pi_*(\tmf)[(\Delta^3)^{-1}]\right)^\hhat_I $$
where $I=(3,c_4)$ (see \cite[\S2]{tmf-book-henriques}). 
It is clear from the calculation of $\pi_*\tmf$ that the localization map $\pi_*(\tmf) \to (\Delta^3)^{-1}\pi_*(\tmf)$ is
an injection.
It suffices to show that completion at $I$ is injective on the specified
classes. This holds because $0=c_4 \cdot \alpha =c_4 \cdot \beta = c_4 \cdot b$ in
$(\Delta^{24})^{-1}\pi_*\tmf$ for degree reasons
(and these classes are also all 3-torsion).

(2) Consider an element of $\ker(L:\Eel_\iy(\tmf)\to \EE_\iy(\TMF))$ represented
by $x\in \Eel_2(\tmf)$ in filtration 0 or 1. We claim that $x$ is in
$\ker(L_2:\Eel_2(\tmf)\to \EE_2(\TMF))$:
since $L_2(x)$ is in filtration 0 or 1, it cannot be the target of a
$d_r$ differential for $r\geq 2$.
By comparing the calculations of $\Eel_2(\tmf)$ and $\EE_2(\TMF/3)$ 
in \cite[\S5]{bauer-tmf} and \cite[Theorem 1.1]{HKM}, respectively, it is clear
that $L'_2: \Eel_2(\tmf/3)\to \EE_2(\TMF/3)$ is an injection, so the image of
$x$ in $\Eel_2(\tmf/3)$ is zero, which implies (using exactness of the top row
in the diagram) $x\in \Eel_2(\tmf)$ is 3-divisible.
$$ \xymatrix{
\Eel_2(\tmf)\ar[r]^-3 & \Eel_2(\tmf)\ar[r]^-i\ar[d]_-{L_2} &
\Eel_2(\tmf/3)\ar@{^(->}[d]^-{L'_2}
\\ & \EE_2(\TMF)\ar[r]^-i & \EE_2(\TMF/3)
}$$
Since $E_2^\el(\tmf)$ has no 3-divisible classes in filtration 1, we now focus
on the filtration 0 case.
Let $y=x/3^n\in \Eel_2(\tmf)$ be
the non-3-divisible generator, which then has nonzero image $i(y)$ in $\Eel_2(\tmf/3)$. Since $L'_2$ is an
injection, $L'_2(i(y))=i(L_2(y))\neq 0$. We claim that the (nonzero) group
generated by $L_2(y)$ is torsion-free:
if not, then the corresponding top cell class would be a nonzero
class in $\EE_2(\TMF/3)$ in filtration $-1$, contradicting \cite[Theorem
1.1]{HKM}. So $L_2(x)=3^n L_2(y)\neq 0$, contradicting the fact above that $x\in
\ker(L_2)$.
\end{proof}

\begin{remark}
Our methods are not sufficient to completely determine the image of the map
$h':\pi_*(S/3)\to \pi_*(\tmf/3)$. The remaining nontrivial part of this question is to
determine which elements $\Delta^n \alpha$ are in the image. Arguments similar
to those we have given in this section show that $h'(\bar{\beta_{9t+2}}) =
\Delta^{6t+1}\alpha$ and $h'(\bar{\beta_{9t+5}}) = \Delta^{6t+3}\alpha$.
However, the families $\Delta^{6t}\alpha$ for $t\geq 1$ and $\Delta^{6t+4}\alpha$
for $t\geq 0$ fit into patterns that are not described by our work
in this paper. For example, $\Delta^4\alpha\in \pi_{99}(\tmf/3)$ is not in the
image of $h'$ for degree reasons. On the other hand, using
the more precise definitions of the $\beta$ elements in \cite[(2.4)]{MRW} and
calculating analogously to Lemma \ref{lem:H(v23)}, we find that the
map $E_2(S/3)\to \EE_2(\TMF/3)$ sends $\bar{\beta_{18/11}}$ to
$\Delta^{10}\alpha$. As we do not know if $\bar{\beta_{18/11}}$ is a permanent
cycle, we are unable to conclude whether $\Delta^{10}\alpha$ is in the image of
$\pi_*(S/3)$.
\end{remark}

\printbibliography

\end{document}